\documentclass[12pt,a4paper,reqno]{amsart}

\usepackage{amsmath,amsthm}
\usepackage{amssymb}
\usepackage{mathtools}
\numberwithin{equation}{section}

\usepackage{latexsym}
\usepackage{hyperref}

\usepackage{cite}
\usepackage{enumitem}

\usepackage[margin=2.5cm]{geometry}
\usepackage{eucal}

\usepackage{tikz-cd}
\usepackage{tikz-3dplot}
\usepackage{pgfplots}
\pgfplotsset{compat=1.18}
\usetikzlibrary{positioning}

\usepackage{todonotes}

\theoremstyle{plain}
\newtheorem{Theorem}{Theorem}[section]
\newtheorem{Proposition}[Theorem]{Proposition}
\newtheorem{Lemma}[Theorem]{Lemma}
\newtheorem{Corollary}[Theorem]{Corollary}

\theoremstyle{definition}
\newtheorem{Remark}[Theorem]{Remark}
\newtheorem{Definition}[Theorem]{Definition}
\newtheorem{Example}[Theorem]{Example}
\newtheorem{Terminology}[Theorem]{Terminology}
\newtheorem{Notation}[Theorem]{Notation}
\newtheorem{Warning}[Theorem]{Warning}
\theoremstyle{remark}

\DeclareMathOperator{\Der}{Der}

\DeclareMathOperator{\Hom}{Hom}
\DeclareMathOperator{\End}{End}
\DeclareMathOperator{\Aut}{Aut}

\DeclareMathOperator{\Span}{Span}

\DeclareMathOperator{\Empty}{{\_\!\_\,}}
\DeclareMathOperator{\Comma}{\downarrow}
\DeclareMathOperator*{\Lim}{lim}

\newcommand{\bbA}{\mathbb{A}}

\newcommand{\bbR}{\mathbb{R}}
\newcommand{\bbZ}{\mathbb{Z}}

\newcommand{\calA}{\mathcal{A}}
\newcommand{\calB}{\mathcal{B}}
\newcommand{\calC}{\mathcal{C}}

\newcommand{\calI}{\mathcal{I}}

\newcommand{\calM}{\mathcal{M}}
\newcommand{\calN}{\mathcal{N}}
\newcommand{\calO}{\mathcal{O}}
\newcommand{\calT}{\mathcal{T}}
\newcommand{\calTcone}{\mathcal{T}^\triangleright}

\newcommand{\calX}{\mathcal{X}}

\newcommand{\catB}{{\mathcal{B}}}
\newcommand{\catC}{{\mathcal{C}}}

\newcommand{\frakg}{\mathfrak{g}}
\newcommand{\op}{{\mathrm{op}}}
\newcommand{\id}{\mathrm{id}}
\newcommand{\Id}{\mathrm{1}}
\newcommand{\pr}{\mathrm{pr}}

\newcommand{\Ad}{\mathrm{Ad}}
\newcommand{\Lie}{\mathcal{L}}

\newcommand{\Val}{\eta}
\newcommand{\Rmult}{\hat{m}}
\newcommand{\Radd}{\hat{+}}
\newcommand{\Runit}{\hat{1}}
\newcommand{\Rzero}{\hat{0}}
\newcommand{\Rminus}{\hat{-}}

\newcommand{\Aadd}{\hat{+}}
\newcommand{\Azero}{\hat{0}}
\newcommand{\Aminus}{\hat{-}}
\newcommand{\Amod}{\hat{\kappa}}

\newcommand{\Tpart}[1]{T_{(#1)}}
\newcommand{\Term}{!}
\newcommand{\bT}{{{}^b T}}
\newcommand{\bcalX}{{{}^b \calX}}
\newcommand{\bMfld}{{{}^b \Mfld}}

\newcommand{\Set}{{\mathcal{S}\mathrm{et}}}
\newcommand{\Cat}{{\mathcal{C}\mathrm{at}}}
\DeclareMathOperator{\intCat}{\underline{\Cat}}
\newcommand{\Eucl}{{\mathcal{E}\mathrm{ucl}}}

\renewcommand{\phi}{\varphi}
\renewcommand{\epsilon}{\varepsilon}

\mathcode`\:=\string"603A

\newcommand{\Mfld}{{\mathcal{M}\mathrm{fld}}}
\newcommand{\Mod}{{\mathcal{M}\mathrm{od}}}
\newcommand{\Aff}{{\mathcal{A}\mathrm{ff}}}

\newcommand{\CAlg}{{\mathrm{C}\mathcal{A}\mathrm{lg}}}
\newcommand{\CinftyRing}{{\mathrm{C^\infty}\mathcal{R}\mathrm{ing}}}

\DeclareMathOperator{\Sym}{Sym}
\newcommand{\SymGr}{S}
\DeclareMathOperator{\Spec}{Spec}
\DeclareMathOperator{\Pro}{Pro}
\DeclareMathOperator{\Ind}{Ind}

\begin{document}

\title{Cartan calculus in tangent categories}

\author[L.~Aintablian]{Lory Aintablian}
\address{Max-Planck-Institut f\"ur Mathematik, Vivatsgasse 7, 53111 Bonn, Germany}
\email{lory.aintablian@gmail.com}

\author[C.~Blohmann]{Christian Blohmann}
\address{Max-Planck-Institut f\"ur Mathematik, Vivatsgasse 7, 53111 Bonn, Germany}
\email{blohmann@mpim-bonn.mpg.de}

\subjclass[2020]{18F40 (58A50, 53D17, 58A40)}

\date{\today}

\keywords{Tangent category, Cartan calculus, abstract Lie algebroid, Lie-Rinehart algebra, Lie-Rinehart forms, vector field, scalar multiplication}

\begin{abstract}
We determine the structure needed in a tangent category in the sense of Rosick\'y and Cockett-Cruttwell to construct the Cartan calculus on all objects. The missing ingredient is a scalar multiplication by a commutative ring object $R$, playing the role of the smooth real line, which equips the tangent bundle of every object with the structure of an $R$-module compatible with the tangent structure. We show that under these axioms the Lie algebra of vector fields acts by derivations on the ring of $R$-valued functions and satisfies the Leibniz rule. In other words, the tangent bundle is an abstract Lie algebroid, so that the Lie algebra of vector fields is a Lie-Rinehart algebra over the ring of functions. Consequently, every object carries a Cartan calculus of Lie-Rinehart forms, given by the Chevalley-Eilenberg complex together with its differential, inner derivative, and Lie derivative. Examples include the tangent categories of smooth manifolds, $G$-manifolds, Lie groupoids, log manifolds, pro-manifolds, elastic diffeological spaces, affine and general schemes, graded manifolds, and affine $C^\infty$-schemes.
\end{abstract}
\maketitle

\setcounter{tocdepth}{1}
\tableofcontents

\section{Introduction}

A considerable part of the local computations in differential geometry is carried out in terms of the Lie algebra of vector fields, the graded commutative algebra of differential forms, the de Rham differential, the inner derivative, and the Lie derivative. On a smooth manifold, these structures together with their algebraic relations form the \emph{Cartan calculus} (Definition~\ref{def:CartanCalc}).

The purpose of this paper is to determine the structure needed in a tangent category to construct the Cartan calculus on all objects. A tangent category, in the sense of Rosick\'y~\cite{Rosicky:1984} and further developed by Cockett and Cruttwell~\cite{CockettCruttwell:2014}, axiomatizes the tangent functor and the natural transformations that are needed to abstractly define the Lie bracket of vector fields. What was missing is a scalar multiplication by a ring object $R$ in the category satisfying the appropriate axioms to allow for the construction of the Cartan calculus.

What are these axioms? It is rather obvious that the tangent bundle $TX \to X$ should be equipped with the structure of an $R$-module in the slice category over $X$. Rosick\'y remarks as much in his original paper~\cite{Rosicky:1984}. It is also clear that we have to require that $TR \cong R \times R$ as $R$-modules, so that we have a projection onto the fiber of $TR \to R$, which is needed for the definition of the differential of a function. Similarly, we should require the vertical lift to be $R$-linear (cf.~\cite{CruttwellLemay:2023}). It is not clear, however, what further compatibility conditions are needed to ensure that vector fields act by derivations on the ring of $R$-valued functions and that the Lie bracket satisfies the Leibniz rule.

The missing axioms for the $R$-module structure $\kappa_X: R \times TX \to TX$ are given by the commutative diagrams~\eqref{diag:ScalarMult2} and \eqref{diag:ScalarMult3}. They relate the tangent map $T\kappa_X$ with $\kappa_X$, the vertical lift, and the symmetric structure on $T^2 X$. This culminates in the definition of a cartesian tangent category with \emph{scalar $R$-multiplication} (Definition~\ref{def:RScalar}). 

We show that the Lie algebra of vector fields $\calX \coloneqq \Gamma(X, TX)$ on an object $X$ acts by derivations on the ring of $R$-valued functions $\calO \coloneqq \Hom(X, R)$ (Proposition~\ref{prop:VecActFunc}). Then we show that the Lie bracket satisfies the Leibniz rule (Proposition~\ref{prop:LeibnizRule}). This is summarized in Theorem~\ref{thm:TangentLieAlgd}, which states that $TX \to X$ is an abstract Lie algebroid in the sense of \cite[Def.~7.1]{AintablianBlohmann:2025}, the \emph{tangent Lie algebroid} of $X$. It follows from Proposition~\ref{prop:LieAlgdRinehart} that $\calX$ is a Lie-Rinehart algebra over $\calO$ (Definition~\ref{def:LieRinehartAlg}). It is well-known (Proposition~\ref{prop:CartanCalcLieRine}) that from a Lie-Rinehart algebra we can construct a Cartan calculus, given by the graded commutative $\calO$-algebra of \emph{Lie-Rinehart forms} 
\begin{equation}
\label{eq:LieRinehartFormsIntro}
  \Omega^\bullet(X) = \Hom_\calO(\wedge^\bullet_\calO \calX, \calO)
\end{equation}
together with the Chevalley-Eilenberg differential, the inner derivative, and the Lie derivative (Theorem~\ref{thm:CartanOnX}). This is the \emph{Cartan calculus} of $X$.

The concept of scalar multiplication of Definition~\ref{def:RScalar} and the results of this paper apply to a wide range of examples (Section~\ref{sec:Examples}), including the Cartan calculus of
\begin{enumerate}

\item smooth vector fields and de Rham forms on a smooth manifold,

\item $G$-equivariant vector fields and forms on a $G$-manifold,

\item multiplicative vector fields and multiplicative forms on a Lie groupoid,

\item log vector fields and log forms on a log manifold,

\item pro-vector fields and ind-forms on a pro-manifold such as the infinite jet manifold,

\item diffeological vector fields and forms on an elastic diffeological space, as they are used in Lagrangian Field Theory,

\item derivations and Lie-Rinehart forms on affine and general schemes,

\item degree zero vector fields and forms on a graded manifold,

\item $C^\infty$-derivations and forms on an affine $C^\infty$-scheme.

\end{enumerate}
The list is not exhaustive and can easily be extended.

\subsection{Future directions}

In several of the examples listed above, the Lie-Rinehart forms as defined in~\eqref{eq:LieRinehartFormsIntro} are not the kind of forms that are usually considered. This is the case for non-smooth schemes and $C^\infty$-schemes, where the Lie-Rinehart forms are not the usual K\"ahler forms, and on elastic diffeological spaces, where the Lie-Rinehart forms are generally not the usual diffeological de Rham forms. It turns out that there is a smaller Cartan calculus given by vector fields and cubical forms (called singular forms in \cite{CruttwellLucyshynWright:2018}), which are the expected forms in these examples. This is the subject of a forthcoming paper. In the case of graded manifolds, there is a different issue: The usual graded Cartan calculus is given by \emph{graded} vector fields in all degrees and forms with the total degree given by the sum of the form degree and the additional grading. To obtain this we have to enrich the tangent structure over graded vector spaces, which is a topic for future research. Finally, we conjecture that the concept of scalar multiplication can be straightforwardly generalized to higher tangent categories, in the sense of \cite{BauerBurkeChing:2021}. The motivating example is the 2-category of differentiable stacks (Section~\ref{sec:DifferentiableStacks}).

\subsection{Related work}

Tangent categories were introduced by Rosick\'y~\cite{Rosicky:1984} and later rediscovered and substantially developed by Cockett and Cruttwell~\cite{CockettCruttwell:2014}, who broadened the framework by replacing abelian groups with commutative monoids. They are closely tied to the differential categories arising from the semantics of differential linear logic. The Cartesian differential categories of Blute, Cockett, and Seely~\cite{BluteCockettSeely:2009} axiomatize the directional derivative of multivariable calculus and coincide with the differential objects of a tangent category, while their integral counterpart was studied by Cockett and Lemay~\cite{CockettLemay:2019}. Cockett, Lemay, and Lucyshyn-Wright~\cite{CockettLemayLucyshynWright:2020} showed that the co-Eilenberg-Moore category of a differential category is a tangent category; in particular, this exhibits the opposite of the category of commutative rings---the basic setting of algebraic geometry---as a tangent category. We build on this result to show that the category of affine schemes is a tangent category with scalar multiplication by the affine line (Section~\ref{sec:AffSchemes}).

Two structural results clarify the meaning of the axioms of a tangent category. Leung~\cite{Leung:2017} classified tangent structures as strong monoidal functors out of a category of Weil algebras, making precise the connection with synthetic differential geometry; building on this, Garner~\cite{Garner:2018} proved that every tangent category embeds into a representable one, in which the tangent functor is given by exponentiation by an infinitesimal object.

A substantial body of work develops differential geometry synthetically within an arbitrary tangent category. This includes vector fields and their Lie bracket, whose Jacobi identity was established by Cockett and Cruttwell~\cite{CockettCruttwell:2015}; differential bundles and fibrations~\cite{CockettCruttwell:2018}, the tangent-categorical analogue of vector bundles, further characterized by MacAdam~\cite{MacAdam:2021} and by Ching~\cite{Ching:2024}; connections, together with curvature, torsion, and the Bianchi identities~\cite{CockettCruttwellConnections:2017}; solutions of differential equations and their flows~\cite{CockettCruttwellLemay:2021}; and differential and sector forms, from which Cruttwell and Lucyshyn-Wright~\cite{CruttwellLucyshynWright:2018} extracted two notions of cohomology.

The framework of tangent categories accommodates a broad range of examples beyond smooth manifolds. Cruttwell and Lemay~\cite{CruttwellLemay:2023} analyzed differential bundles in commutative algebra and algebraic geometry, identifying them with modules and quasi-coherent sheaves; Ikonicoff, Lanfranchi, and Lemay~\cite{IkonicoffLanfranchiLemay:2023} exhibited categories of algebras over an operad as tangent categories, and Lanfranchi studied the resulting operadic geometry~\cite{Lanfranchi:2023b} as well as a Grothendieck construction for tangent fibrations~\cite{Lanfranchi:2023a}. Vooys~\cite{Vooys:2023} showed that ind-categories of tangent categories are again tangent categories. In the higher-categorical direction, Bauer, Burke, and Ching~\cite{BauerBurkeChing:2021} extended the theory to $\infty$-categories, making precise the analogy between Goodwillie's calculus of functors and the differential calculus of manifolds.

Most closely related to the present paper are developments concerning Lie theory in tangent categories. Burke and MacAdam~\cite{BurkeMacAdam:2019} introduced involution algebroids as a tangent-categorical generalization of Lie algebroids. In~\cite{AintablianBlohmann:2025} we defined differentiable groupoid objects and their abstract Lie algebroids, the notion we use here. While vector fields and forms have each been considered separately, the full algebraic structure of the Cartan calculus has, to the best of our knowledge, not previously been constructed in the framework of tangent categories.

\subsection{Outline of the paper}

Section~\ref{sec:AbstractTangent} recalls the categorical framework of abstract tangent structures and Rosick\'{y}'s axioms. In Section~\ref{sec:CartesianTangent}, we show that the definition of cartesian tangent categories can be reduced to the simple requirement that the map $T(X \times Y) \to TX \times TY$ is a natural isomorphism (Proposition~\ref{prop:CartTanPreserve}). In Section~\ref{sec:ScalarMult} we introduce the scalar $R$-multiplication with its axioms, our main new ingredient. In Section~\ref{sec:VecActFunc} we prove that vector fields act on scalar functions and that their Lie bracket obeys the Leibniz rule. Section~\ref{sec:Cartan} assembles these results into the Cartan calculus, relating abstract Lie algebroids, Lie-Rinehart algebras, and Cartan calculi, and establishing the main theorems. The paper closes in Section~\ref{sec:Examples} with the examples listed above.

\subsection{Notation}

We follow \cite{Rosicky:1984} for the notation for the composition of functors and natural transformations. The composition of functors $G: \calA \to \calB$ and $F: \calB \to \calC$ will be denoted by juxtaposition $FG: \calA \to \calC$. Proceeding iteratively, the $n$-fold composition of an endofunctor $F: \calC \to \calC$ with itself is also denoted by $F^n \coloneqq F \ldots F$.

The horizontal composition of natural transformations $\alpha: F \to F'$ and $\beta: G \to G'$ will also be denoted by juxtaposition $\alpha\beta: FG \to F'G'$. Its components are
\begin{equation*}
  (\alpha \beta)_A
  = \alpha_{G'A} \circ F(\beta_A)
  = F'(\beta_A) \circ  \alpha_{GA}
  \,.
\end{equation*}
We will use the notation $F$ for the identity natural transformation $\id_F: F \to F$. We have the relations
\begin{align*}
  (F\beta)_A &= F(\beta_A)
  \\
  (\alpha G)_A &= \alpha_{GA}
  \,.
\end{align*}
The vertical composition of $\alpha$ with a natural transformation $\alpha': F' \to F''$ will be denoted by $\alpha' \circ \alpha: F \to F''$. Its components are given by $(\alpha' \circ \alpha)_B = \alpha'_B \circ \alpha_B$ for all $B \in \catB$. The strict monoidal category of endofunctors on $\calC$ will be denoted by $\End(\calC)$ and the identity endofunctor by $\Id: \calC \to \calC$.

\subsection*{Acknowledgements}

We thank David Aretz, Dennis Borisov, Michael Ching, Geoffrey Cruttwell, Madeleine Jotz, Benjamin MacAdam, David Mi\-ya\-mo\-to, Leonid Ryvkin, and Tashi Walde for helpful comments and fruitful discussions. L.A.~was funded by the Hausdorff Center for Mathematics in Bonn.

\section{Abstract tangent structures}
\label{sec:AbstractTangent}

\subsection{Bundles with algebraic structure}
\label{sec:BunAbelGp}

In category theory and related areas of mathematics, the most basic meaning of ``bundle'' is a morphism $A \to X$. The morphism is often required to have additional good properties and the fibers are equipped with additional structure. This generalizes various kinds of fiber bundles in geometry. The algebraic structures we will consider are given by finite product algebraic theories, which we recall.

\begin{Definition}
\label{def:TAlg}
Let $\calT$ be a small category with finite products (including a terminal object as empty product). A \textbf{$\calT$-algebra} in a category $\calC$ is a functor $\calT \to \calC$ that preserves finite products. A \textbf{morphism of $\calT$-algebras} is a natural transformation that preserves finite products. 
\end{Definition}

Examples of algebraic structures that are $\calT$-algebras in the sense of Definition~\ref{def:TAlg} are semigroups, monoids, commutative monoids, groups, abelian groups, semirings, rigs, rings, $R$-modules, $k$-vector spaces, $R$-algebras, and $C^\infty$-rings. When $\calT$ is the theory of, say, abelian groups, a $\calT$-algebra in $\calC$ is called an abelian group \textbf{internal to} $\calC$ or \textbf{in} $\calC$ for short.

A $\calT$-algebra $A: \calT \to \calC_{/X}$ in the slice category over $X \in \calC$ will be called a \textbf{bundle of $\calT$-algebras over $X$}. If the pullback $A_x \coloneqq * \times_X A: \calT \to \calC$ over a point $x: * \to X$ exists in $\calC$, then $A_x$ is a $\calT$-algebra in $\calC$, called the \textbf{fiber over $x$}.

When $\calT$ is the theory of $k$-vector spaces, a bundle of $\calT$-algebras is called a bundle of $k$-vector spaces. This definition makes no assumptions about local trivializations whatsoever, so a bundle of vector spaces is a priori not a vector bundle in the usual sense. Therefore, the following fact is somewhat surprising and very useful.

\begin{Proposition}[Example~2.7 in \cite{AintablianBlohmann:2025}]
Bundles of $\bbR$-vector spaces in the category of smooth finite-dimensional manifolds are vector bundles in the usual sense (with local trivializations).
\end{Proposition}

We will consider morphisms between bundles of $\calT$-algebras with different base objects, so we do not always want to fix the base object. Instead, we incorporate the bundle structure in the algebraic theory of bundles of $\calT$-algebras as follows. 

\begin{Notation}
The \textbf{right cone} of $\calT$ will be denoted by $\calTcone$ \cite[Notation~1.2.8.4]{Lurie:2009}. The tip of the cone will be denoted by $*$, since it is the terminal object in $\calTcone$.
\end{Notation}

%

\begin{Definition}
\label{def:BundleOfTAlgs}
A \textbf{bundle of $\calT$-algebras} in the category $\calC$ is a functor $A: \calTcone \to \calC$ that preserves the limits of the diagrams $t \to * \leftarrow t'$ for all $t, t' \in \calT$. A morphism of bundles of $\calT$-algebras is a natural transformation that intertwines those limits.
\end{Definition}

Let us spell out Definition~\ref{def:BundleOfTAlgs} in more detail. Let $A: \calTcone \to \calC$ be a functor. Let us denote the image of an arrow $t \to *$ with $t \in \calT$ by
\begin{equation*}
  A(t \to *) = (A_t \xrightarrow{p_t} X)
  \,.
\end{equation*}
Since $*$ is the terminal object, $\Lim(t \rightarrow * \leftarrow t') \cong t \times t'$, so there is a natural morphism
\begin{equation}
\label{eq:ProdToFibProd}
  A_{t \times t'} \longrightarrow
  A_t \times_X^{p_t, p_{t'}}\! A_{t'}
  \,.
\end{equation}
$A$ is a bundle of $\calT$-algebras if and only if~\eqref{eq:ProdToFibProd} is an isomorphism for all $t, t' \in \calT$. Let
\begin{align*}
  B: \calTcone 
  &\longrightarrow \calC
  \\
  (t \to *) 
  &\longmapsto (B_t \xrightarrow{q_t} Y)   
\end{align*}
be another bundle of $\calT$-algebras. A natural transformation $\phi: A \to B$ is a morphism of bundles of $\calT$-algebras if and only if the square
\begin{equation*}
\begin{tikzcd}[column sep=3.5em]
A_{t \times t'}
\ar[r, "\phi_{t\times t'}"]
\ar[d]
&
B_{t \times t'}
\ar[d]
\\
A_t \times_X^{p_t, p_{t'}}\! A_{t'}
\ar[r, "\phi_t \times_{\phi_*} \phi_{t'}"']
&
B_t \times_Y^{q_t, q_{t'}}\! B_{t'}
\end{tikzcd}
\end{equation*}
commutes for all $t, t' \in \calT$.

\begin{Example}
\label{ex:BundNoAlg}
Let $\calT = *$ be the terminal category. A $\calT$-algebra in $\calC$ is an object $* \to \calC$ without further structure. The right cone is $\calTcone \cong [1]$, the category with two objects and one non-identity morphism. A bundle of $\calT$-algebras is an arrow $A \to X$ in $\calC$. A morphism of bundles is a commutative square:
\begin{equation}
\label{eq:MorphBund1}
\begin{tikzcd}
A \ar[r, "\phi"] \ar[d, "p"'] & A' \ar[d, "p'"]
\\
X \ar[r, "f"'] & X' 
\end{tikzcd}    
\end{equation}
The category of bundles of $\calT$-algebras is the functor category $\calC^{[1]} \equiv \intCat([1], \calC)$.
\end{Example}

\begin{Example}
\label{ex:BundOfAb}
A bundle of abelian groups consists of a morphism $p: A \to X$ together with the morphisms
\begin{equation*}
\begin{tikzcd}[column sep={tiny}]
A \times_X A \ar[rr, "+"] \ar[dr, "p \, \circ \, \pr_1 = p \, \circ \, \pr_2"'] & & A \ar[dl, "p"] \\
& X &
\end{tikzcd}
\qquad
\begin{tikzcd}[column sep={tiny}]
X \ar[rr, "0"] \ar[dr, "\id_X"'] & & A \ar[dl, "p"] \\
& X &
\end{tikzcd}
\qquad
\begin{tikzcd}[column sep={tiny}]
A \ar[rr, "i"] \ar[dr, "p"'] & & A \ar[dl, "p"] \\
& X &
\end{tikzcd}
\end{equation*}
of the addition, the zero, and the inverse, satisfying the usual axioms of an abelian group.

Let $p': A' \to X'$ be another bundle of abelian groups with structure morphisms $(+', 0', i')$. A morphism $A \to A'$ is a commutative square~\eqref{eq:MorphBund1} such that the diagram
\begin{equation}
\label{eq:MorphBund3}
\begin{tikzcd}
A \times_X A \ar[r, "\phi \times_f \phi"] 
\ar[d, "+"'] & A' \times_{X'} A' \ar[d, "+'"]
\\
A \ar[r, "\phi"'] & A'
\end{tikzcd}    
\end{equation}
commutes. As for ordinary groups, it is implied that $(\phi,f)$ intertwines the units, $\phi\, \circ\, 0 = 0' \circ f$. Composing this equation with $p'$ on the left, we obtain $p' \circ \phi \circ 0 = f$, which shows that $f$ is uniquely determined by $\phi$. Therefore, we can denote a morphism~\eqref{eq:MorphBund1} of bundles of abelian groups by $\phi$ alone.    
\end{Example}

\begin{Example}
\label{ex:BundlesTalgExamples1}
Let $\calT$ be a small category with finite products; let $\calC$ be a category with finite limits; let $A, B: \calTcone \to \calC$ be bundles of $\calT$-algebras, denoted by $A(t \to *) = (A_t \xrightarrow{p_t} X)$, $B(t \to *) = (B_t \xrightarrow{q_t} Y)$.
\begin{itemize}

\item[(i)] The constant functor $X: \calTcone \to \calC_{/X}$, $(t \to *) \mapsto (X \xrightarrow{\id} X)$ is trivially a bundle of $\calT$-algebras.

\item[(ii)] Let $* \in \calC$ be a terminal object. Since $\calC_{/*} \cong \calC$, every bundle of $\calT$-algebras over $*$ can be viewed equivalently as a $\calT$-algebra.

\item[(iii)] Assume that the bundles $A$ and $B$ have the same base $X = Y$. Then the functor $A \times_X B: \calTcone \to \calC$ defined by
\begin{equation*}
  (A \times_X B)(t \to *)
  \coloneqq
  (A_t \times_X^{p_t, q_t} B_t \to X)
\end{equation*}
is a bundle of $\calT$-algebras over $X$, called the \textbf{fiber product} of $A$ and $B$.

\item[(iv)] The functor $A\times B: \calTcone \to \calC$ defined by
\begin{equation*}
  (A \times B)(t \to *)
  \coloneqq
  (A_t \times B_t \xrightarrow{p_t \times q_t} X \times Y)
\end{equation*}
is a bundle of $\calT$-algebras over $X \times Y$, called the \textbf{external product} of $A$ and $B$.

\item[(v)] As a special case of Example~(iv), we can consider $A = X$ from Example~(i) and $B: \calT \to \calC \cong \calC_{/*}$ a $\calT$-algebra in $\calC$, viewed as a bundle over the terminal object $*$ as in Example~(ii). The external product $X \times B$ is the \textbf{trivial bundle} of $\calT$-algebras over $X$ with fiber $B$.

\item[(vi)] Let $Y \to X$ be a morphism in $\calC$. The pullback functor $Y \times_X \Empty: \calC_{/X} \to \calC_{/Y}$, $(C \to X) \mapsto (Y \times_X C \to Y)$ preserves products. It follows that the composition $(Y \times_X \Empty) \circ A: \calT \to \calC_{/Y}$ is a $\calT$-algebra, which we denote by $Y \times_X A$ and call the \textbf{pullback bundle} of $\calT$-algebras.

\end{itemize}
\end{Example}

\begin{Example}
Let $R$ be a commutative ring object in the category $\calC$ with addition $\Radd$, zero $\Rzero$, multiplication $\Rmult$, and unit $\Runit$. Let $X \in \calC$ be an object. Let $X \times R$ denote the trivial bundle of rings of Example~\ref{ex:BundlesTalgExamples1}~(v). An $(X \times R)$-module in $\calC_{/X}$ will be called, for short, a \textbf{bundle of $R$-modules} over $X$. Spelled out, this is a bundle of abelian groups $p:A \to X$ together with a morphism 
\begin{equation}
\label{diag:ModStr1}
\begin{tikzcd}[column sep={tiny}]
R \times A 
\ar[rr, "\kappa"] 
\ar[dr, "p \circ \pr_2"'] &&
A \ar[dl, "p"]
\\
& X &
\end{tikzcd}
\end{equation}
satisfying the usual conditions of a linear action of a ring. The axioms are spelled out in Proposition~\ref{prop:RModuleBundle} for the case of the tangent bundle $TX \to X$ in a tangent category.

Let $p:A \to X$ and $p': A' \to X'$ be bundles of $R$-modules with $R$-actions $\kappa$ and $\kappa'$. A morphism of bundles of $R$-modules is a morphism $\phi: A \to A'$ of bundles of abelian groups such that the diagram
\begin{equation}
\label{eq:MorphBundR2}
\begin{tikzcd}[column sep=large]
R \times A \ar[r, "\id_R \times \phi"] 
\ar[d, "\kappa"'] & R \times A' \ar[d, "\kappa'"]
\\
A \ar[r, "\phi"'] & A'
\end{tikzcd}    
\end{equation}
commutes.    
\end{Example}

\subsection{Sections of bundles}

The functor of sections
\begin{equation}
\label{eq:SectionFunctor}
\begin{aligned}
  \Gamma: \calC^{[1]} 
  &\longrightarrow \Set
  \\
  (A \xrightarrow{p} X)
  &\longmapsto
  \Gamma(X,A) 
  = \{a:X \to A \ | \ p \circ a = \id_X \}  
\end{aligned}
\end{equation}
is monoidal, mapping the fiber product of bundles to the cartesian product,
\begin{equation*}
  \Gamma(X, A \times_X B) 
  \cong
  \Gamma(X,A) \times \Gamma(X,B)
  \,.
\end{equation*}
This implies that for every bundle of $\calT$-algebras $A: \calTcone \to \calC$ over $X$, the functor
\begin{align*}
  \Gamma(X,A): \calT &\longrightarrow \Set
  \\
  t &\longmapsto \Gamma(X, A_t)
\end{align*}
is a $\calT$-algebra.

\begin{Example}
Let $p: A \to X$ be a bundle of abelian groups in $\calC$ with structure morphisms $(+,0,i)$. Then $\Gamma(X,A)$ is an abelian group in the usual sense. The sum of two sections $a,b: X \to A$ is given by
\begin{equation*}
  a+b \coloneqq + \circ (a,b)
  \,.
\end{equation*}  
The zero section $0: X \to A$ is the zero of the group and $i \circ a$ is the inverse.
\end{Example}

\begin{Example}
\label{ex:PointsOfR}
Let $R$ be a ring object in $\calC$. The set of sections of the trivial bundle of rings $X \times R \to X$ is a ring, which is naturally isomorphic via
\begin{equation}
\label{eq:SectionsTrivial}
\begin{aligned}
  \Gamma(X, X \times R) &\xrightarrow{~\cong~}
  \Hom(X, R) \\
  a &\longmapsto \pr_2 \circ a
\end{aligned}
\end{equation}
to the ring $\Hom(X,R)$ of $R$-valued functions with pointwise addition and multiplication
\begin{equation*}
  f + g \coloneqq \Radd \circ (f,g)
  \,,\quad
  fg \coloneqq \Rmult \circ (f,g)
  \,.
\end{equation*}
The zero of this ring is $X \to * \stackrel{\Rzero}{\to} R$, the multiplicative unit is $X \to * \stackrel{\Runit}{\to} R$.
\end{Example}

\begin{Example}
Let $A \to X$ be a bundle of $R$-modules with $R$-action $\kappa: R \times A \to A$. Using the isomorphism \eqref{eq:SectionsTrivial}, we see that $\Gamma(X,A)$ has the structure of a $\Hom(X,R)$-module given by 
\begin{equation}
\label{eq:SectionsModuleStr}
  fa \coloneqq \kappa \circ (f,a) \,,
\end{equation}
for all $f \in \Hom(X,R)$ and $a \in \Gamma(X,A)$. Moreover, if $\phi: A \to A'$ is a morphism of bundles of $R$-modules with base map $\id_X$, then $\phi_*$ is a morphism of modules over $\Hom(X,R)$.    
\end{Example}

\subsection{Symmetric structure on an endofunctor}




\begin{Definition}
\label{def:SymStruc}
Let $F: \calC \to \calC$ be a functor and $\tau: F^2 \to F^2$ a natural transformation. Let $\tau_{12} \coloneqq \tau\,F$ and $\tau_{23}\coloneqq F\,\tau$ be the two trivial extensions of $\tau$ to natural transformations $F^3 \to F^3$. We call $\tau$ a \textbf{braiding on $F$} if it satisfies the braid relations $\tau_{12} \circ \tau_{23} \circ \tau_{12} = \tau_{23} \circ \tau_{12} \circ \tau_{23}$. A braiding $\tau$ is called a \textbf{symmetric structure on $F$} if it satisfies $\tau\circ \tau = F^2$.
\end{Definition}

\begin{Remark}
A symmetric structure $\tau$ on a functor $F:\calC \to \calC$ defines an action of the symmetric group $S_n$ on $F^n$. The action of a transposition $\tau_{i,i+1}$, $1 \leq i < n$ is given by the natural isomorphism
\begin{equation*}
  F^{i-1} \tau F^{n-i-1}: F^n \longrightarrow F^n \,.
\end{equation*}
Since the symmetric group on $n$ elements is generated by adjacent transpositions, this induces a group homomorphism $S_n \to \Aut(F^n)$, where $\Aut(F^n)$ is the group of natural isomorphisms $F^n \to F^n$.
\end{Remark}

\begin{Definition}
\label{def:PreserveFibProd}
An endofunctor $F: \calC \to \calC$ \textbf{preserves the fiber products} of a bundle $p:A \to X$ if for all $k \geq 1$ the natural morphism of bundles over $FX$,
\begin{equation}
\label{eq:MorphBundles4}
  \nu_{k,X}:
  F(A \times_X^{p,p} \ldots \times_X^{p,p} A) \longrightarrow
  FA \times_{FX}^{Fp, Fp} \ldots \times_{FX}^{Fp, Fp} FA
  \,,
\end{equation}
where both sides have the same number $k$ of factors, is an isomorphism.
\end{Definition}


\subsection{Rosick\'{y}'s axioms}
\label{sec:RosickysAxioms}

In~\cite{Rosicky:1984}, Rosick\'{y} introduced the notion of abstract tangent functors, which axiomatizes the natural categorical structure of the tangent functor of manifolds that is needed to define the Lie bracket of vector fields. 

\begin{Definition}[Sec.~2 in \cite{Rosicky:1984}, Def.~2.3 in \cite{CockettCruttwell:2014}]
\label{def:TangentStructure}
A \textbf{Rosick\'y tangent structure} on a category $\calC$ consists of a functor $T: \calC \to \calC$ together with natural transformations $\pi: T \to \Id$, $0: \Id \to T$, $+: T_2 \to T$, $\lambda: T \to T^2$, and $\tau: T^2 \to T^2$, such that the following axioms hold:

\textbf{(i)} The \textbf{fiber products}
\begin{equation}
\label{eq:TanFun1}
  T_k \coloneqq \underbrace{T \times_\Id T \times_\Id \ldots \times_\Id T}_{k \text{ factors}}
\end{equation}
over $T \stackrel{\pi}{\to} \Id$ exist for all $k \geq 1$, are pointwise, and preserved by $T$ (Definition~\ref{def:PreserveFibProd}).

\textbf{(ii)} $T \stackrel{\pi}{\to} \Id$ with neutral element $0$ and addition $+$ is a \textbf{bundle of abelian groups} over $\Id$ (Definition~\ref{def:BundleOfTAlgs} and Example~\ref{ex:BundOfAb}).

\textbf{(iii)}
$\tau: T^2 \to T^2$ is a \textbf{symmetric structure} on $T$ (Definition~\ref{def:SymStruc}) and a morphism of bundles of abelian groups, that is,
\begin{equation}
\label{eq:TanFun2}
\begin{tikzcd}[column sep=tiny]
T^2 \ar[rr, "\tau"] \ar[dr, "T \pi"'] && T^2 \ar[dl, "\pi T"]
\\
& T &
\end{tikzcd}    
\qquad
\begin{tikzcd}[column sep=large]
T^2 \times_T^{T\pi, T\pi} T^2
\ar[r, "\tau \times_T \tau"] 
\ar[d, "\nu_2^{-1}"']
&
T_2 T
\ar[dd, "+T"] 
\\
T T_2  \ar[d, "T+"']
&
\\
T^2 \ar[r, "\tau"'] 
&
T^2
\end{tikzcd}     
\end{equation}
commutes, where $\nu_2$ is the morphism~\eqref{eq:MorphBundles4} for $A=TX \xrightarrow{\pi_X} X$, $F=T$, and $k=2$.


\textbf{(iv)} The \textbf{vertical lift} $\lambda$ is a morphism of bundles of abelian groups, that is,
\begin{equation}
\label{eq:TanFun3}
\begin{tikzcd}
T \ar[r, "\lambda"] \ar[d, "\pi"'] & T^2 \ar[d, "\pi T"]
\\
\Id \ar[r, "0"'] & T
\end{tikzcd}    
\qquad\qquad
\begin{tikzcd}
T \ar[r, "\lambda"] \ar[d, "\lambda"'] & 
T^2 \ar[d, "\lambda T"]
\\
T^2 \ar[r, "T \lambda"'] & T^3
\end{tikzcd}    
\end{equation}
commutes, and $(+T) \circ (\lambda \times_0 \lambda) = \lambda \circ +$.

\textbf{(v)} The diagrams
\begin{equation}
\label{eq:TanFun5}
\begin{tikzcd}[column sep=tiny]
& T \ar[dl, "\lambda"'] \ar[dr, "\lambda"] &
\\
T^2 \ar[rr, "\tau"'] & &
T^2
\end{tikzcd}    
\qquad\qquad
\begin{tikzcd}
T^2 \ar[r, "T\lambda"] \ar[d, "\tau"'] & 
T^3 \ar[r, "\tau T"] &
T^3 \ar[d, "T\tau"]
\\
T^2 \ar[rr, "\lambda T"'] & &
T^3
\end{tikzcd}    
\end{equation}
commute.

\textbf{(vi)} The diagram
\begin{equation}
\label{eq:TanFun4}
\begin{tikzcd}
T \ar[r, "\lambda"] \ar[d, "\pi"'] 
\ar[dr, phantom, "\lrcorner", very near start] 
& T^2 \ar[d, "{(\pi T, T\pi)}"]
\\
\Id \ar[r, "{(0, 0)}"'] & T_2
\end{tikzcd}    
\end{equation}
is a pointwise pullback square. 
\end{Definition}

\begin{Terminology}
\label{term:negatives}
In \cite{CockettCruttwell:2014} and subsequent work, Rosick\'{y}'s original requirement that $T \to \Id$ be a bundle of abelian groups is relaxed to a bundle of abelian monoids. The authors later refer to Rosick\'{y}'s original notion as tangent categories with negatives \cite{CockettCruttwell:2015,CockettCruttwellConnections:2017}. We will follow a more recent terminology and call these \textbf{Rosick\'y tangent categories} \cite{CruttwellLemay:2023,IkonicoffLanfranchiLemay:2023,LanfranchiLemay:2025}.
\end{Terminology}

The vertical lift can be extended by the additive bundle structure to the map
\begin{equation*}
  \lambda_2:
  T_2 \xrightarrow{~T0 \times_0 \lambda~}
  T_2 T
  \xrightarrow{~+T~}
  T^2
  \xrightarrow{~\tau~}
  T^2
  \,.
\end{equation*}
In components, 
\begin{equation}
\label{eq:VertLiftExt}
\begin{split}
  (\lambda_2)_X
  &= \tau_X \circ {+_{TX}} \circ (T0_X \times_{0_X} \lambda_X)
  \,,
\end{split}
\end{equation}
for all $X \in \calC$. It was shown in \cite[Lem.~3.10]{CockettCruttwell:2014}, assuming all other axioms of a Rosick\'y tangent structure, that Axiom~\eqref{eq:TanFun4} is satisfied if and only if
\begin{equation}
\label{diag:lambda2pullback}
\begin{tikzcd}
T_2 \ar[r, "\lambda_2"] \ar[d, "\pi \circ \pr_1"'] 
\ar[dr, phantom, "\lrcorner", very near start] 
& 
T^2 \ar[d, "{T\pi}"]
\\
\Id \ar[r, "0"'] & T
\end{tikzcd}    
\end{equation}
is a pointwise pullback. It is straightforward to show that $\lambda_2$ is linear in the second argument, that is, the diagram
\begin{equation}
\label{diag:lambda2linear}
\begin{tikzcd}[column sep=large]
T_2 X \times_{TX}^{\pr_1, \pr_1} T_2 X
\ar[r, "\cong"]
\ar[d, "(\lambda_2)_X \, \times_{TX} \, (\lambda_2)_X"']
&
T_3 X 
\ar[r, "\id_{TX} \, \times_X \, (+_X)"]
&[2em]
T_2 X
\ar[d, "(\lambda_2)_X"]
\\
T^2 X \times_{TX}^{\pi_{TX}, \pi_{TX}} T^2 X
\ar[rr, "+_{TX}"']
&&
T^2 X
\end{tikzcd}
\end{equation}
commutes for all $X \in \calC$. 


\begin{Definition}
Let $\calC$ be a category with a tangent structure. A \textbf{vector field} on $X \in \calC$ is a section of $\pi_X: TX \to X$.
\end{Definition}

The natural transformation that maps a pair of tangent vectors in the same fiber to their difference will be denoted by
\begin{equation}
\label{eq:Difference}
  -: T_2 \xrightarrow{~(T \times i)~} 
  T_2 \xrightarrow{~+~} T
  \,,
\end{equation}
where $i: T \to T$ is the additive inverse of the bundle of abelian groups. The bracket of two vector fields $v,w: X \to TX$ is defined as follows: The morphism
\begin{equation}
\label{eq:DeltaOrig}
  \delta(v,w)
  \coloneqq 
  -_{TX} \circ 
  (Tw \circ v, \tau_X \circ Tv \circ w): X \longrightarrow T^2X
\end{equation}
satisfies
\begin{equation}
\label{eq:DeltaKerOrig}
\begin{aligned}
  \pi_{TX} \circ \delta(v,w) &= w
  T\pi_X \circ \delta(v,w) &= 0_X
  \,.
\end{aligned}
\end{equation}
It then follows from the universal property of the pullback that there is a unique vector field $[v,w]$ such that
\begin{equation}
\label{eq:deltaBracketRel}
  \delta(v,w) = (\lambda_2)_X \circ \bigl(w, [v,w] \bigr)
  \,.
\end{equation}
It is straightforward to check that $[v,w]$ is additive in $v$ and $w$. It was announced in \cite{Rosicky:1984} and proved in \cite{CockettCruttwell:2015} with the input of Rosick\'y that $[v,w]$ satisfies the Jacobi identity. Observe that all ingredients of the tangent structure, including the additive inverse of $T \to \Id$, are needed for the definition of the bracket of vector fields.

\section{Cartesian tangent structures}
\label{sec:CartesianTangent}

\subsection{Cartesian tangent structures}

We recall that a functor $T: \calC \to \calC$ is said to \textbf{preserve finite products} if the natural morphism
\begin{equation}
\label{eq:chiXY}
  \chi_{X,Y} \coloneqq (T\pr_1, T\pr_2): T(X \times Y) 
  \longrightarrow TX \times TY
\end{equation}
has an inverse for all $X,Y \in \calC$.

\begin{Definition}
\label{def:CartTan}
A tangent structure on $\calC$ is called \textbf{cartesian} if the tangent functor preserves finite products.
\end{Definition}

\begin{Definition}
\label{def:MorTanCat}
Let $(\calC, T, \pi, 0, +, \tau, \lambda)$ and $(\calC', T', \pi', 0', +', \tau', \lambda')$ be tangent categories. A \textbf{strong morphism of Rosick\'y tangent categories} is a functor $F: \calC \to \calC'$ together with a natural isomorphism
\begin{equation*}
  \alpha: FT \longrightarrow T'F
\end{equation*}
that satisfies the following conditions:
\begin{itemize}

\item[(i)] $F$ preserves all finite fiber products of the tangent bundle, that is, for all $k \geq 1$, the morphism
\begin{equation*}
  \nu_k: F(T_k) \longrightarrow (FT)_k
  \,,
\end{equation*}
where
\begin{equation*}
  (FT)_k =
  FT \times_F^{F\pi,F\pi} \ldots \times_F^{F\pi,F\pi} FT
\end{equation*}
with $k$ factors, has an inverse (Def.~\ref{def:PreserveFibProd} for any functor).

\item[(ii)] $\alpha$ is a morphism of bundles of abelian groups over $F$, that is, the diagrams
\begin{equation}
\label{diag:Strong1}
\begin{tikzcd}[column sep=tiny]
FT \ar[rr, "\alpha"] \ar[dr, "F\pi"'] 
&& 
T'F \ar[dl, "\pi'F"]
\\
& F &
\end{tikzcd}
\qquad\qquad
\begin{tikzcd}[column sep=3em]
F T_2  
\ar[r, "\nu_2"]
\ar[d, "F+"'] 
&
(FT)_2
\ar[r, "\alpha \times_F \alpha"] 
&
T'_2 F \ar[d, "+'F"]
\\
FT \ar[rr, "\alpha"'] 
&& 
T' F
\end{tikzcd}    
\end{equation}
commute.

\item[(iii)] $\alpha$ intertwines the symmetric structures and the vertical lifts, that is, the diagrams
\begin{equation}
\label{diag:Strong2}
\begin{tikzcd}[column sep=3.5em]
FT^2 
\ar[r, "T' \alpha \, \circ \, \alpha T"]
\ar[d, "F \tau"']
&
T'^2 F
\ar[d, "\tau' F"]
\\
F T^2 
\ar[r, "T' \alpha \, \circ \, \alpha T"']
&
T'^2 F
\end{tikzcd}
\qquad\qquad
\begin{tikzcd}[column sep=3.5em]
FT 
\ar[r, "\alpha"]
\ar[d, "F\lambda"']
&
T'F
\ar[d, "\lambda' F"]
\\
F T^2 
\ar[r, "T' \alpha \, \circ \, \alpha T"']
&
T'^2 F
\end{tikzcd}
\end{equation}
commute.

\end{itemize}
\end{Definition}

\begin{Remark}
\label{rem:StrongComparison}
Definition~\ref{def:MorTanCat} is very similar to Definition~2.7 of \cite{CockettCruttwell:2014}, but there are a number of differences. The first, most substantial difference is that we require in (i) that $F$ preserves only the finite fiber products of the tangent bundle. In \cite{CockettCruttwell:2014} the requirement was that ``$F$ preserves the equalizers and pullbacks of the tangent structure''. It is not clear to us which equalizers and pullbacks, exactly, were meant. Moreover, in Proposition~\ref{prop:MorphTanStruc}, we show that $F$ always preserves the pullback square of the vertical lift. The second difference is that the top horizontal arrow $(\alpha \times_F \alpha) \circ \nu_2$ of the diagram involving the addition was denoted in \cite{CockettCruttwell:2014} as ``$\alpha_2 F$'', a notation that was not clear to us. We believe that the right diagram of~\eqref{diag:Strong1} clarifies what ``$\alpha_2 F$'' must have meant. The third difference is that we only give a definition of strong morphisms, since we do not need the weaker notion of morphism given in \cite{CockettCruttwell:2014}. The fourth difference is that in a Rosick\'y tangent category the tangent bundle is a bundle of abelian groups. As a consequence, the commutativity of
\begin{equation}
\label{diag:Strong3}
\begin{tikzcd}[column sep=tiny]
& F \ar[dl, "F0"'] \ar[dr, "0'F"] &
\\
FT \ar[rr, "\alpha"'] & &
T' F
\end{tikzcd}    
\end{equation}
is implied by the compatibility with the addition, so that it need not be included as an axiom. For the more general notion of tangent structure without negatives, this diagram has to be added to Definition~\ref{def:MorTanCat}.
\end{Remark}

\begin{Example}
The functor $T:\calC \to \calC$ together with the natural isomorphism given by the symmetric structure $\tau:T^2 \to T^2$ is a strong morphism from the tangent category $\calC$ to itself.
\end{Example}

\begin{Proposition}
  \label{prop:MorphTanStruc}
  Let $(\calC \xrightarrow{\: F \:} \calC', FT \xrightarrow{\: \alpha \:} T'F)$ be a strong morphism of tangent categories (Definition~\ref{def:MorTanCat}). Then the diagram
  \begin{equation*}
  \begin{tikzcd}[column sep=3em]
  FT 
  \ar[r, "F\lambda"] 
  \ar[d, "F\pi"'] 
  & 
  FT^2 \ar[d, "{F(\pi T, T\pi)}"]
  \\
  F \ar[r, "{F(0, 0)}"'] & FT_2
  \end{tikzcd}    
  \end{equation*}
  is a pointwise pullback.
\end{Proposition}  

\begin{Lemma}
  \label{lem:InnerOuterSquares}
  Consider the following diagram:
  \begin{equation*}
  \begin{tikzcd}[column sep=0.8em, row sep=0.8em]
  A 
  \ar[rrr] 
  \ar[ddd] 
  \ar[dr, "\alpha"]
  & & &
  B \ar[ddd] 
  \ar[dl]
  \\
  &
  A' 
  \ar[d] 
  \ar[r] 
  &
  B' \ar[d]
  &
  \\
  &
  C' \ar[r] &
  D'
  &
  \\
  C \ar[rrr]
  \ar[ur] 
  & & &
  D 
  \ar[ul, "\delta"']
  \end{tikzcd}
  \end{equation*}
  Assume that the four outer trapezoids commute. If the inner square commutes and $\delta$ is a monomorphism, then the outer square commutes. Dually, if the outer square commutes and $\alpha$ is an epimorphism, then the inner square commutes.
\end{Lemma}
\begin{proof}
  See Lemma~A.5 in \cite{AintablianBlohmann:2025}.
\end{proof}

\begin{proof}[Proof of Proposition~\ref{prop:MorphTanStruc}]
Consider the following diagram:
\begin{equation*}
\begin{tikzcd}[column sep=3em]
FT 
\ar[rrr, "F\lambda"] 
\ar[ddd, "F\pi"'] 
\ar[dr, "\alpha"]
& & &
FT^2 \ar[ddd, "{F(\pi T, T\pi)}"] 
\ar[dl, "T' \alpha \, \circ \, \alpha T"']
\\
&
T'F 
\ar[d, "\pi'F"'] 
\ar[r, "\lambda'F"] 
&
T'^2 F \ar[d, "{(\pi' T'F, T'\pi' F)}"]
&
\\
&
F \ar[r, "{(0'F, 0'F)}"'] &
T'_2 F
&
\\
F \ar[rrr, "{F(0,0)}"'] 
\ar[ur, "\id"'] 
& & &
FT_2 
\ar[ul, "{(\alpha \times_F \alpha) \, \circ \, \nu_2}"]
\end{tikzcd}    
\end{equation*}
The left trapezoid is the commutative triangle at the left of~\eqref{diag:Strong1}. The upper trapezoid is the commutative square at the right of~\eqref{diag:Strong2}. The lower trapezoid commutes since 
\begin{equation*}
\begin{split}
(\alpha \times_F \alpha) \circ \nu_2 \circ F(0,0)
&=
(\alpha \times_F \alpha) \circ (F0,F0)
\\
&=
(\alpha \circ F0, \alpha \circ F0)
\\
&=
(0'F,0'F) \,,
\end{split}
\end{equation*}
where in the last step we have used Diagram~\eqref{diag:Strong3}. The right trapezoid commutes since
\begin{equation*}
\begin{split}
(\alpha \times_F \alpha) \circ \nu_2 \circ F(\pi T, T \pi)
&=
(\alpha \times_F \alpha) \circ (F\pi T, FT\pi)
\\
&=
(\alpha \circ F\pi T, \alpha \circ FT\pi)
\\
&=
(\alpha \circ \pi' FT \circ \alpha T, T'F \pi \circ \alpha T)
\\
&=
(\alpha \circ \pi' FT, T'F \pi) \circ \alpha T
\\
&=
(\pi' T' F \circ T' \alpha, T' \pi' F \circ T' \alpha) \circ \alpha T
\\
&=
(\pi' T' F, T' \pi' F) \circ T' \alpha \circ  \alpha T \,,
\end{split}
\end{equation*}
where we have used the left commutative triangle in \eqref{diag:Strong1}, the naturality of $\alpha$, the naturality of $\pi'$, and once more the commutativity of the left triangle in \eqref{diag:Strong1}. The inner square evaluated on $X \in \calC$ is the pointwise pullback~\eqref{eq:TanFun4} of the vertical lift $\lambda'$ evaluated on $FX$. It follows from Lemma~\ref{lem:InnerOuterSquares} that the outer square commutes. Moreover, since all diagonal maps are isomorphisms and the inner square is a pointwise pullback, so is the outer square.
\end{proof}

If $\calC$ and $\calC'$ are tangent categories, then $\calC \times \calC'$ has a tangent structure given by the products of the tangent functors and the natural transformations.


\begin{Proposition}
\label{prop:CartTanPreserve}
Let $\calC$ be a category with finite products and a tangent structure. The following are equivalent:
\begin{itemize}

\item[(i)] The tangent structure is cartesian.

\item[(ii)] The natural transformation $\chi_{X,Y}: T(X \times Y) \to TX \times TY$ is a natural isomorphism, such that the product functor $\calC \times \calC \to \calC$, $(X,Y) \mapsto X \times Y$ together with $\alpha \coloneqq \chi^{-1}$ is a strong morphism of tangent categories.

\end{itemize}
\end{Proposition}
  
\begin{proof}
Condition~(ii) assumes that $\chi$ is a natural isomorphism, so that (ii) implies (i).

To show that (i) implies (ii), we have to show that if $\chi$ is a natural isomorphism, then all conditions of Definition~\ref{def:MorTanCat} are satisfied for $\alpha = \chi^{-1}$. First, we observe that, since pullbacks commute with products, the morphism
\begin{equation*}
\begin{split}
  (\nu_2)_{X, Y}:
  T_2 X \times T_2 Y
  &=
  (TX \times_X TX) \times (TY \times_Y TY)
  \\
  &\stackrel{\cong}{\longrightarrow}
  (TX \times TY) \times_{X \times Y} (TX 
  \times TY)    
\end{split}
\end{equation*}
is an isomorphism. It follows by induction that Definition~\ref{def:MorTanCat}~(i) is satisfied. 

Due to the naturality of the bundle projection, we have the commutative diagram
\begin{equation*}
\begin{tikzcd}
TX \ar[d, "\pi_X"'] 
& 
T(X \times Y) 
\ar[l, "T\pr_1"'] \ar[r, "T\pr_2"]
\ar[d, "\pi_{X \times Y}"]
&
TY \ar[d, "\pi_Y"]
\\
X 
& 
X \times Y \ar[l, "\pr_1"] \ar[r, "\pr_2"']
&
Y
\end{tikzcd}
\end{equation*}
This implies that 
\begin{equation}
\label{eq:ProdPresMorph2a}
\begin{tikzcd}[column sep=-0.5em]
T(X \times Y) \ar[rr, "\chi_{X,Y}"] \ar[dr, "\pi_{X \times Y}"'] 
&& 
TX \times TY \ar[dl, "\pi_X \times \pi_Y"]
\\
& X \times Y &
\end{tikzcd}
\end{equation}
commutes by the universal property of products. By inverting the horizontal arrow in~\eqref{eq:ProdPresMorph2a}, we obtain the left diagram of Definition~\ref{def:MorTanCat}~(ii). 

Moreover, we see that $\chi_{X,Y}$ induces a morphism of the fiber products
\begin{equation*}
  (\chi \times_F \chi)_{X, Y}:
  T_2 (X \times Y)
  \longrightarrow
  (TX \times TY) \times_{X \times Y}
  (TX \times TY) \,.
\end{equation*}
From the naturality of the addition we obtain the commutative diagram
\begin{equation*}
\begin{tikzcd}
T_2 X \ar[d, "+_X"'] 
& 
T_2 (X \times Y) 
\ar[l, "T_2 \pr_1"'] \ar[r, "T_2 \pr_2"]
\ar[d, "+_{X \times Y}"]
&
T_2 Y \ar[d, "+_Y"]
\\
TX 
& 
T(X \times Y) \ar[l, "T\pr_1"] \ar[r, "T\pr_2"']
&
TY
\end{tikzcd}
\end{equation*}
which implies that the diagram
\begin{equation}
\label{eq:ProdPresMorph2b}
\begin{tikzcd}
T_2(X \times Y)
\ar[r] \ar[d, "+_{X \times Y}"']
&
T_2 X \times T_2 Y
\ar[d, "+_X \times +_Y"]
\\
T(X \times Y)
\ar[r, "\chi_{X,Y}"']
&
TX \times TY    
\end{tikzcd}
\end{equation}
commutes. It follows from the universal property of the product that the top horizontal arrow is the morphism
\begin{equation*}
\begin{split}
  T_2 (X \times Y)
  &\xrightarrow{~(\chi \times_F \chi)_{X,Y}~}
  (TX \times TY) \times_{X \times Y} (TX \times TY)
  \\
  &\xrightarrow{~(\nu_2^{-1})_{X,Y}~}
  T_2 X \times T_2 Y
  \,. 
\end{split}
\end{equation*}
This shows that by inverting the top and bottom horizontal arrows of~\eqref{eq:ProdPresMorph2b}, we obtain the second diagram of Definition~\ref{def:MorTanCat}~(ii).

From the naturality of the symmetric structure we obtain the commutative diagram
\begin{equation*}
\begin{tikzcd}
T^2 X \ar[d, "\tau_X"'] 
& 
T^2 (X \times Y) 
\ar[l, "T^2 \pr_1"'] \ar[r, "T^2 \pr_2"]
\ar[d, "\tau_{X \times Y}"]
&
T^2 Y \ar[d, "\tau_Y"]
\\
T^2 X 
& 
T^2 (X \times Y) \ar[l, "T^2 \pr_1"] \ar[r, "T^2 \pr_2"']
&
T^2 Y
\end{tikzcd}
\end{equation*}
which implies that the diagram
\begin{equation}
\label{eq:ProdPresMorph3a}
\begin{tikzcd}
T^2(X \times Y)
\ar[r] \ar[d, "\tau_{X \times Y}"']
&
T^2 X \times T^2 Y
\ar[d, "\tau_X \times \tau_Y"]
\\
T^2 (X \times Y)
\ar[r]
&
T^2 X \times T^2 Y    
\end{tikzcd}
\end{equation}
commutes. It follows from the universal property of the product that the horizontal arrows are given by the morphism
\begin{equation}
\label{eq:T2Prod}
  T^2 (X \times Y)
  \xrightarrow{~T\chi_{X,Y}~}
  T(TX \times TY)
  \xrightarrow{~\chi_{TX,TY}~}
  T^2 X \times T^2 Y
  \,. 
\end{equation}
This shows that by inverting the top and bottom horizontal arrows of~\eqref{eq:ProdPresMorph3a}, we obtain the first diagram of Definition~\ref{def:MorTanCat}~(iii).

Finally, from the naturality of the vertical lift we obtain the commutative diagram
\begin{equation*}
\begin{tikzcd}
TX \ar[d, "\lambda_X"'] 
& 
T(X \times Y) 
\ar[l, "T \pr_1"'] \ar[r, "T \pr_2"]
\ar[d, "\lambda_{X \times Y}"]
&
TY \ar[d, "\lambda_Y"]
\\
T^2 X 
& 
T^2 (X \times Y) \ar[l, "T^2 \pr_1"] \ar[r, "T^2 \pr_2"']
&
T^2 Y
\end{tikzcd}
\end{equation*}
which implies that the diagram
\begin{equation}
\label{eq:ProdPresMorph3b}
\begin{tikzcd}
T(X \times Y)
\ar[r, "\chi_{X,Y}"] \ar[d, "\lambda_{X \times Y}"']
&
TX \times TY
\ar[d, "\lambda_X \times \lambda_Y"]
\\
T^2 (X \times Y)
\ar[r]
&
T^2 X \times T^2 Y    
\end{tikzcd}
\end{equation}
commutes. It follows from the universal property of the product that the bottom horizontal arrow is the same morphism as~\eqref{eq:T2Prod}. By inverting the top and bottom horizontal arrows of~\eqref{eq:ProdPresMorph3b}, we obtain the second diagram of Definition~\ref{def:MorTanCat}~(iii). We conclude that under the assumption that $\chi$ is a natural isomorphism, all the conditions of Definition~\ref{def:MorTanCat} are satisfied, which finishes the proof.
\end{proof}

\begin{Remark}
Proposition~\ref{prop:CartTanPreserve} can be refined as follows. Assume that the product functor $\calC \times \calC \to \calC$, $(X,Y) \mapsto X \times Y$ together with some natural isomorphism $\alpha_{X,Y}: TX \times TY \to T(X \times Y)$ is a strong morphism of tangent categories. From the naturality of $\alpha^{-1}$ we obtain the commutative diagram
\begin{equation*}
\begin{tikzcd}
TX \ar[d, "\alpha^{-1}_{X,*}"'] 
& 
T(X \times Y) 
\ar[l, "T\pr_1"'] \ar[r, "T\pr_2"]
\ar[d, "\alpha^{-1}_{X \times Y}"]
&
TY \ar[d, "\alpha^{-1}_{*,Y}"]
\\
TX 
& 
TX \times TY \ar[l, "\pr_1"] \ar[r, "\pr_2"']
&
TY
\end{tikzcd}
\end{equation*}
which implies that the diagram
\begin{equation}
\label{eq:ProdPresMorph0}
\begin{tikzcd}[column sep=-1.0em]
T(X \times Y)
\ar[rr, "\chi_{X,Y}"] \ar[dr, "\alpha^{-1}_{X \times Y}"']
&&
TX \times TY
\ar[dl, "\alpha^{-1}_{X, *} \times \alpha^{-1}_{*,Y}"]
\\
& 
TX \times TY
&
\end{tikzcd}
\end{equation}
commutes. This shows that
\begin{equation*}
  \chi_{X,Y} = (\alpha_{X,*} \times \alpha_{*,Y}) 
  \circ \alpha^{-1}_{X,Y}
  \,,
\end{equation*}
so that $\chi_{X,Y}$ is an isomorphism. We conclude that if the product functor is a strong morphism of tangent categories at all, then the natural transformation must be $\chi$.    
\end{Remark}

\begin{Remark}
Proposition~\ref{prop:CartTanPreserve} shows that Definition~\ref{def:CartTan} of cartesian tangent structures is equivalent to Definition~2.8 in \cite{CockettCruttwell:2014}.
\end{Remark}

\subsection{Partial derivatives}

It was observed in \cite[Def.~2.9]{CockettCruttwell:2014} that when the tangent structure on $\calC$ is cartesian, we can define the \textbf{partial tangent morphisms} of a morphism $f: X \times Y \to Z$ in $\calC$ by
\begin{align}
  \Tpart{1} f: TX \times Y 
  &\xrightarrow{~\id_{TX} \times 0_Y~}
  TX \times TY \xrightarrow{~\chi_{X,Y}^{-1}~}
  T(X \times Y) \xrightarrow{~Tf~} 
  TZ
  \\
  \label{eq:PartialT2}
  \Tpart{2} f: X \times TY 
  &\xrightarrow{~0_X \times \id_{TY}~}
  TX \times TY \xrightarrow{~\chi_{X,Y}^{-1}~}
  T(X \times Y)\xrightarrow{~Tf~} 
  TZ
  \,,
\end{align}
where the index refers to the factor in the product and where $\chi$ is the natural transformation~\eqref{eq:chiXY}. In \cite[Def.~2.9]{CockettCruttwell:2014} the partial tangent morphisms were labelled by the factors $X$ and $Y$. We use positional indices in order to describe cases where $X = Y$.

The partial tangent morphisms are functorial in $X$, $Y$, and $Z$. That is, if there is a commutative diagram
\begin{equation*}
\begin{tikzcd}
X \times Y 
\ar[r, "f"]
\ar[d, "\alpha \times \beta"']
&
Z
\ar[d, "\gamma"]
\\
X' \times Y' 
\ar[r, "f'"']
&
Z'
\end{tikzcd}
\end{equation*}
then the diagram
\begin{equation}
\label{eq:T2functorial}
\begin{tikzcd}[column sep=large]
X \times TY 
\ar[r, "\Tpart{2} f"]
\ar[d, "\alpha \times T\beta"']
&
TZ
\ar[d, "T\gamma"]
\\
X' \times TY' 
\ar[r, "\Tpart{2} f'"']
&
TZ'
\end{tikzcd}
\end{equation}
commutes. Since the tangent morphism is linear, $\Tpart{2}f$ is linear in the second argument. That is, the diagram
\begin{equation}
\label{eq:T2linear}
\begin{tikzcd}[column sep=5em]
X \times T_2 Y 
\ar[r, "\cong"]
\ar[d, "\id_X \times +_Y"']
&
(X \times TY) \times_{X \times Y} (X \times TY)
\ar[r, "{\Tpart{2}f \, \times_f \, \Tpart{2} f}"] 
&
T_2 Z
\ar[d, "+_Z"]
\\
X \times TY 
\ar[rr, "\Tpart{2} f"']
&&
TZ
\end{tikzcd}
\end{equation}
commutes. The analogous statements hold for $\Tpart{1} f$. From the partial tangent morphisms we can reconstruct the tangent morphism by the following result, which is straightforward to prove.

\begin{Proposition}[Prop.~2.10 in \cite{CockettCruttwell:2014}]
\label{prop:TSumPartials}
Let $f: X \times Y \to Z$ be a morphism of $\calC$. Then
\begin{equation*}
  Tf = +_Z \circ 
  \bigl(\Tpart{1} f \circ 
  (\id_{TX} \times \pi_Y),
  \Tpart{2} f \circ
  (\pi_X \times \id_{TY})
  \bigr) \circ
  \chi_{X,Y}
  \,.
\end{equation*}
\end{Proposition}

\section{Scalar multiplication}
\label{sec:ScalarMult}

\subsection{Module structure}
\label{sec:ModStructure}

Let $R$ be a ring object in the category $\calC$. It gives rise to an endofunctor $\Id \times R: \calC \to \calC$, $X \mapsto X \times R$, which is equipped with the projection $\pr_1: \Id \times R \to \Id$. The ring structure of $R$ equips $\Id \times R \to \Id$ with the structure of a ring internal to endofunctors over $\Id$. An $(\Id \times R \to \Id)$-module in $\End(\calC) \Comma \Id$ will be called a \textbf{bundle of $R$-modules} over $\Id$ (Definition~\ref{def:BundleOfTAlgs}). An $(\Id \times R \to \Id)$-module structure on a bundle $T \to \Id$ will be called, for short, an \textbf{$R$-module structure} on $T \to \Id$. For clarity and later reference we spell out this structure explicitly.

\begin{Proposition}
\label{prop:RModuleBundle}
Let $\calC$ be a tangent category; let $R$ be a unital ring object in $\calC$ with addition $\Radd$, zero $\Rzero$, multiplication $\Rmult$, and unit $\Runit$. An $R$-module structure on the bundle $\pi:T \to \Id$ is given by a natural morphism
\begin{equation*}
  \kappa_X: R \times TX 
  \longrightarrow TX
  \,,
\end{equation*}
such that the following diagrams commute for all $X \in \calC$:
\begin{itemize}

\item[(i)] Morphism of bundles:
\begin{equation*}
\begin{tikzcd}[column sep={tiny}]
R \times TX \ar[rr, "\kappa_X"] \ar[dr, "\pi_X \circ \, \pr_2"'] &&
TX \ar[dl, "\pi_X"]
\\
& X &
\end{tikzcd}
\end{equation*}

\item[(ii)] Associativity:
\begin{equation*}
\begin{tikzcd}[column sep={large}]
R \times R \times TX \ar[r, "\id_R \times \kappa_X"] 
\ar[d, "\Rmult \times \id_{TX}"'] &
R \times TX \ar[d, "\kappa_X"]
\\
R \times TX \ar[r, "\kappa_X"'] & TX
\end{tikzcd}
\end{equation*}

\item[(iii)] Unitality:
\begin{equation*}
\begin{tikzcd}[column sep=3em]
* \times TX \ar[r, "\Runit \times \id_{TX}"] \ar[dr, "\cong"'] & 
R \times TX \ar[d, "\kappa_X"]
\\
& TX
\end{tikzcd}    
\end{equation*}

\item[(iv)] Linearity in $R$:
\begin{equation*}
\begin{tikzcd}[column sep={large}, row sep=3em]
R \times R \times TX
\ar[r, "\Radd \, \times \, \id_{TX}"] 
\ar[d, "{\big( \kappa_X \circ (\pr_1, \pr_3), \, \kappa_X \circ (\pr_2, \pr_3) \big)}"'] &
R \times TX 
\ar[d, "\kappa_X"]
\\
TX \times_X TX 
\ar[r, "+_X"'] 
&
TX
\end{tikzcd}
\end{equation*}

\item[(v)] Linearity in $TX$:
\begin{equation*}
\begin{tikzcd}[column sep={large}, row sep=3em]
R \times TX \times_X TX
\ar[r, "\id_R \, \times \, +_X"] 
\ar[d, "{\bigl( \kappa_X \circ (\pr_1, \pr_2), \, \kappa_X \circ (\pr_1, \pr_3) \bigr)}"'] &
R \times TX 
\ar[d, "\kappa_X"]
\\
TX \times_X TX 
\ar[r, "+_X"'] 
&
TX
\end{tikzcd}
\end{equation*}
\end{itemize}
\end{Proposition}

\begin{Remark}
As is the case for any module structure, linearity in $R$, expressed by diagram~(iv), implies that the scalar multiplication by $\Rzero: * \to R$ sends $TX$ to the zero section, that is, the diagram
\begin{equation}
\label{diag:RzeroPreserve}
\begin{tikzcd}[column sep=4em,row sep=2em]
* \times TX \ar[r, "{\Rzero \times \id_{TX}}"] \ar[d, "{\cong}"'] & R \times TX \ar[dd, "{\kappa_{X}}"] \\
TX \ar[d, "\pi_X"'] & \\
X \ar[r, "0_X"'] & TX
\end{tikzcd}   
\end{equation}        
is commutative. Similarly, linearity in $TX$, expressed by diagram~(v), implies that the scalar multiplication sends the zero section to the zero section, that is, the diagram
\begin{equation}
\label{diag:RzeroPreserve2}
\begin{tikzcd}[column sep=4em,row sep=2em]
R \times X \ar[r, "{\id_R \times 0_{X}}"] \ar[d, "{\pr_2}"'] & R \times TX \ar[d, "{\kappa_{X}}"] 
\\
X \ar[r, "0_X"'] & TX
\end{tikzcd}   
\end{equation}
commutes.
\end{Remark}

\begin{Terminology}
  A tangent category with an $R$-module structure on the tangent bundle $T \to \Id$ will be called, for short, a tangent category with an $R$-module structure. Proposition~\ref{prop:ScalarImpliesNegatives} implies that the tangent category is a Rosick\'y tangent category.
\end{Terminology}

\begin{Proposition}
\label{prop:ScalarImpliesNegatives}
Every tangent category with an $R$-module structure is a Rosick\'y tangent category.
\end{Proposition}

\begin{proof}
This follows from the fact that a commutative monoid with a compatible action of a ring is an abelian group. In the category of sets, we have for every element $x$ of a commutative monoid
\begin{equation*}
\begin{split}
  x + (-\Runit) \cdot x
  &= \Runit \cdot x + (\Rminus \Runit \cdot x)
  = \bigl( \Runit + (\Rminus \Runit) \bigr) \cdot x
  = \Rzero \cdot x
  \\
  &= 0
  \,,
\end{split}
\end{equation*}
which shows that $-x \coloneqq (\Rminus\Runit \cdot x)$ is the additive inverse of $x$. It is straightforward to internalize this statement to an arbitrary category.
\end{proof}

Proposition~\ref{prop:ScalarImpliesNegatives} shows that if we want to work in tangent categories without negatives, the ring $R$ has to be replaced by a rig.


\begin{Proposition}
\label{prop:RCartTan}
In a cartesian tangent category with an $R$-module structure $\kappa_X: R \times TX \to TX$, the diagram
\begin{equation*}
\begin{tikzcd}[column sep=4em, row sep=3em]
R \times T(X \times Y) 
\ar[r, "\id_R \times \chi_{X,Y}"]
\ar[d, "\kappa_{X \times Y}"']
&
R \times TX \times TY
\ar[d, "{\big( \kappa_X \circ (\pr_1, \pr_2), \, \kappa_Y \circ (\pr_1, \pr_3) \big)}"]
\\
T(X \times Y) \ar[r, "\chi_{X,Y}"'] 
&
TX \times TY
\end{tikzcd}
\end{equation*}
commutes for all $X, Y \in \calC$.
\end{Proposition}
\begin{proof}
By the naturality of $\kappa$, the diagram
\begin{equation*}
\begin{tikzcd}[column sep=4em, row sep=3em]
R \times T(X \times Y) 
\ar[r, "\id_R \times T\pr_1"]
\ar[d, "\kappa_{X \times Y}"']
&
R \times TX
\ar[d, "\kappa_X"]
\\
T(X \times Y) \ar[r, "T\pr_1"'] 
&
TX
\end{tikzcd}
\end{equation*}
commutes. Analogously, $T\pr_2 \circ \kappa_{X \times Y} = \kappa_Y \circ (\id_R \times T\pr_2)$. With this, we obtain
\begin{equation*}
\begin{split}
  \chi_{X,Y} \circ \kappa_{X \times Y}
  &=
  (T\pr_1, T\pr_2) \circ \kappa_{X \times Y}
  \\
  &=
  (T\pr_1 \circ \kappa_{X \times Y} , T\pr_2 \circ \kappa_{X \times Y})
  \\
  &=
  \bigl( 
  \kappa_X \circ (\id_R \times T\pr_1), 
  \kappa_Y \circ (\id_R \times T\pr_2) 
  \bigr)
  \\
  &=
  \bigl( 
  \kappa_X \circ (\id_R \times (\pr_1 \circ \chi_{X,Y})),
  \kappa_Y \circ (\id_R \times (\pr_2 \circ \chi_{X,Y}))
  \bigr)
  \\
  &=
  \bigl( 
  \kappa_X \circ (\pr_1, \pr_2), 
  \kappa_Y \circ (\pr_1, \pr_3) 
  \bigr) \circ (\id_R \times \chi_{X,Y})
  \,,
\end{split}
\end{equation*}
which proves the proposition.
\end{proof}

In other words, Proposition~\ref{prop:RCartTan} states that $\chi_{X,Y}: T(X \times Y) \to TX \times TY$ is an isomorphism of bundles of $R$-modules over $X \times Y$ if we equip $TX \times TY$ with the diagonal $R$-module structure.

\begin{Warning}
In the category of sets, an abelian group is the same thing as a $\bbZ$-module. For an abelian group internal to a category $\calC$, this is only true if $\bbZ$ is naturally a ring object in $\calC$, which is not always the case. For example, a category with finite biproducts contains no other ring object than the zero ring. It follows that the only possible tangent structure with an $R$-module structure on such a tangent category is the trivial tangent structure \cite[Section~8.10.3]{AintablianBlohmann:2025}.
\end{Warning}

\subsection{Tangent bundle of a group}

Throughout this section, we assume that $\calC$ is a cartesian tangent category with an $R$-module structure. Let $X$ be an object in $\calC$. Let $x: * \to X$ be a point in $X$. The pullback
\begin{equation*}
  T_x X \coloneqq * \times_X^{x,\pi_X} TX
  \,,
\end{equation*}
if it exists, is called the \textbf{tangent fiber} over $x$. The inclusion of the fiber will be denoted by
\begin{equation*}
  j_x : T_x X \longrightarrow TX
  \,.
\end{equation*}
The addition of tangent vectors, the zero section, and the $R$-module structure restrict to $T_x X$, that is, we have morphisms
\begin{align*}
  +_x: T_x X \times T_x X &\longrightarrow
  T_x X
  \\
  0_x: * &\longrightarrow
  T_x X
  \\
  \kappa_x: R \times T_x X &\longrightarrow
  T_x X
  \,,
\end{align*}
such that the diagrams
\begin{equation}
\label{eq:FiberRMod1}
\begin{tikzcd}
T_x X \times T_x X
\ar[r, "+_x"]
\ar[d, "j_x \times j_x"']
&
T_x X
\ar[d, "j_x"]
\\
TX \times_X TX
\ar[r, "+_X"']
&
TX
\end{tikzcd}
\qquad \qquad
\begin{tikzcd}
*
\ar[r, "0_x"]
\ar[d, "x"']
&
T_x X
\ar[d, "j_x"]
\\
X
\ar[r, "0_X"']
&
TX
\end{tikzcd}
\end{equation}
\begin{equation}
\begin{tikzcd}
R \times T_x X
\ar[r, "\kappa_x"]
\ar[d, "\id_R \times j_x"']
&
T_x X
\ar[d, "j_x"]
\\
R \times TX
\ar[r, "\kappa_X"']
&
TX
\end{tikzcd}    
\end{equation}
commute. These morphisms endow $T_x X$ with the structure of an $R$-module.

\begin{Remark}
Unlike in \cite[Def.~4.13]{CockettCruttwell:2014}, we do not make the assumption that the tangent functor preserves the pullback defining the tangent fibers, since this would exclude interesting applications such as to diffeological spaces.
\end{Remark}

A \textbf{trivialization} of the tangent bundle of an object $X$ in $\calC$ is an isomorphism
\begin{equation*}
  TX \cong X \times A
\end{equation*}
of bundles of $R$-modules over $X$, where $A$ is an $R$-module in $\calC$. By taking the fiber at a point $x: * \to X$, we obtain an isomorphism of $R$-modules $T_x X \cong A$.

\begin{Proposition}
\label{prop:TGtrivial}
Let $(G,m,e)$ be a group object in $\calC$; assume that the pullback
\begin{equation*}
  \frakg \coloneqq T_e G = * \times_{G}^{e,\pi_G} TG
\end{equation*}
exists. Then the tangent bundle of $G$ has a natural trivialization
\begin{equation*}
  TG \cong G \times \frakg
  \,.
\end{equation*}
\end{Proposition}

\begin{proof}
Since the tangent functor is assumed to preserve products, $TG$ is a group object with multiplication
\begin{equation*}
  Tm \circ \chi_{G,G}^{-1}: 
  TG \times TG \longrightarrow TG
\end{equation*}
and identity $Te: * \to TG$, where we have used that $T* \cong *$ since $T$ preserves products. The inverse of the tangent group is the tangent $Ti: TG \to TG$ of the group inverse $i: G \to G$.

The partial tangent morphism of $m$ on the second factor,
\begin{equation*}
  \Tpart{2} m: G \times TG \longrightarrow TG
\end{equation*}
is a left group action of $G$ on $TG$. That is, it satisfies
\begin{subequations}
\begin{align}
  \Tpart{2} m \circ (\id_G \times \Tpart{2} m) 
  &= \Tpart{2} m \circ (m \times \id_{TG})
  \\
  \Tpart{2} m \circ (e \times \id_{TG})
  &= \id_{TG}
  \label{eq:T2mActB}
  \,.
\end{align}
\end{subequations}
The group action has an inverse given by $\Tpart{2} m \circ (i \times \id_{TG})$, which is invertible in the sense that
\begin{equation*}
\begin{split}
  {}&
  \Tpart{2} m \circ 
  (i \times \id_{TG}) \circ 
  (\id_G \times \Tpart{2} m) \circ 
  (\Delta_G \times \id_{TG})
  \\
  ={}&
  \Tpart{2} m \circ (\id_G \times \Tpart{2} m) \circ 
  \bigl( (i, \id_G) \times \id_{TG} \bigr)
  \\
  ={}&
  \Tpart{2} m \circ (m \times \id_TG) \circ 
  \bigl( (i, \id_G) \times \id_{TG} \bigr)
  \\
  ={}&
  \Tpart{2} m \circ \bigl( (e \circ \Term_G) \times \id_{TG} \bigr)
  \\
  ={}&
  \Term_G \times \id_{TG}
  \,,    
\end{split}
\end{equation*}
where $\Term_G: G \to *$ is the terminal morphism.

We can use the left action to left translate the fiber at the group identity to $TG$,
\begin{equation*}
  \phi_G:
  G \times T_e G \xrightarrow{~ \id_G \times j_e~} 
  G \times TG \xrightarrow{~\Tpart{2} m~}
  TG
  \,,
\end{equation*}
where $j_e: T_e G \to TG$ is the inclusion. We can left translate $TG$ back to the fiber at the group identity by 
\begin{equation*}
  \Val'_G:
  TG 
  \xrightarrow{~ \Delta_{TG} ~}
  TG \times TG
  \xrightarrow{~(i \circ \pi_G) \times \id_{TG}}
  G \times TG 
  \xrightarrow{~\Tpart{2} m~} TG
  \,.
\end{equation*}
We have
\begin{equation*}
\begin{split}
  \pi_G \circ \Val'_G
  &= 
  \pi_G \circ \Tpart{2} m \circ 
  \bigl((i \circ \pi_G) \times \id_{TG} \bigr) \circ
  \Delta_{TG}
  \\
  &= 
  m \circ (\id_G \times \pi_G) \circ 
  \bigl((i \circ \pi_G) \times \id_{TG} \bigr) \circ
  \Delta_{TG}
  \\
  &= 
  m \circ (i \times \id_G) \circ 
  (\pi_G \times \pi_G) \circ
  \Delta_{TG}
  \\
  &= 
  m \circ (i \times \id_G) \circ 
  \Delta_G \circ \pi_G
  \\
  &= 
  e \circ \Term_G \circ \pi_G
  \,,
\end{split}
\end{equation*}
From the universal property of the pullback $* \times_G^{e, \pi_G} TG$, it follows that $\Val'_G$ restricts on the codomain to a morphism 
\begin{equation}
\label{eq:ValGDef}
  \Val_G: TG \longrightarrow T_e G
  \,.
\end{equation}
Let 
\begin{equation*}
  \bar{\phi}_G \coloneqq (\pi_G \times \Val_G) \circ \Delta_{TG}:
  TG \to G \times T_e G
  \,.
\end{equation*}
We have
\begin{equation*}
\begin{split}
  \phi_G \circ \bar{\phi}_G
  &=
  \Tpart{2} m \circ (\id_G \times j_e) \circ 
  (\pi_G \times \Val_G) \circ 
  \Delta_{TG}
  \\
  &=
  \Tpart{2} m \circ 
  (\pi_G \times \Val'_G) \circ 
  \Delta_{TG}
  \\
  &=
  \Tpart{2} m \circ 
  \bigl( \pi_G \times (\Tpart{2} m \circ 
  ((i \circ \pi_G) \times \id_{TG}) \circ \Delta_{TG}) 
  \bigr) \circ 
  \Delta_{TG}
  \\
  &=
  \Tpart{2} m \circ (\id_G \times \Tpart{2} m) \circ 
  \bigl( \id_G \times i \times \id_{TG} \bigr) \circ 
  (\pi_G \times \pi_G \times \id_{TG}) \circ 
  \Delta^2_{TG}
  \\
  &=
  \Tpart{2} m \circ (m \times \id_{TG}) \circ 
  \bigl( \id_G \times i \times \id_{TG} \bigr) \circ 
  (\pi_G \times \pi_G \times \id_{TG}) \circ 
  \Delta^2_{TG}
  \\
  &=
  \Tpart{2} m \circ \bigl(
  (m \circ (\id_G \times i) \circ \Delta_G) 
  \times \id_{TG}) \circ \Delta_{TG}
  \\
  &=
  \Tpart{2} m \circ \bigl(
  (e \circ \Term_G) 
  \times \id_{TG}) \circ \Delta_{TG}
  \\
  &=
  \id_{TG}
  \,,
\end{split}
\end{equation*}
where $\Delta^2 = (\Id\circ \Delta) \circ \Delta = (\Delta \circ \Id) \circ \Delta$ and where we have used that $\Tpart{2} m$ is a left group action. With an analogous computation we show that $\bar{\phi}_G \circ \phi_G = \id_G \times \id_{T_e G}$. We conclude that $\bar{\phi}_G$ is the inverse of $\phi_G$.

It follows from Proposition~\ref{prop:RCartTan} that $\Tpart{2} m$ is a morphism of bundles of $R$-modules over $G$. By definition of the $R$-module structure on $T_e G$, the inclusion $j_e: T_e G \to TG$ is a morphism of bundles of $R$-modules. It follows that $\phi_G$ is a morphism of bundles of $R$-modules over $G$. This implies that the inverse of $\phi_G$ is a morphism of bundles of $R$-modules.
\end{proof}

The tangent group structure can be expressed in terms of the trivialized bundle $G \times \frakg$ in the usual way. First we observe that we also have a right action of $G$ on $TG$ given by $\Tpart{1} m: TG \times G \to TG$. The left and right action commute, since
\begin{equation*}
\begin{split}
  \Tpart{1} m \circ (\Tpart{2} m \times \id_G)
  &=
  Tm \circ 
  (Tm \times \id_G) \circ
  \chi_{G,G,G}^{-1} \circ 
  (0_G \times \id_{TG} \times \id_G)
  \\
  &=
  Tm \circ 
  (Tm \times \id_{TG}) \circ
  \chi_{G,G,G}^{-1} \circ 
  (0_G \times \id_{TG} \times 0_G)
  \\
  &=
  Tm \circ 
  (\id_{TG} \times Tm) \circ
  \chi_{G,G,G}^{-1} \circ 
  (0_G \times \id_{TG} \times 0_G)
  \\
  &=
  \Tpart{2} m \circ (\id_G \times \Tpart{1} m)
  \,.
\end{split}
\end{equation*}
The left adjoint action of $G$ on $TG$ is given by
\begin{equation*}
  \Ad: G \times TG 
  \xrightarrow{~(\pr_1, \pr_2, i \circ \pr_1)~}
  G \times TG \times G \xrightarrow{~\Tpart{2} m \times \id_G~}
  TG \times G \xrightarrow{~\Tpart{1} m~}
  TG
  \,.
\end{equation*}
It is straightforward to check that the adjoint action restricts to $T_e G$, that is, there is a commutative diagram
\begin{equation}
\begin{tikzcd}
G \times T_e G
\ar[r, "\Ad_e"]
\ar[d, "\id_G \times j_e"']
&
T_e G
\ar[d, "j_e"]
\\
G \times TG
\ar[r, "\Ad"']
&
TG
\end{tikzcd}
\end{equation}
which equips $\frakg = T_e G$ with a left $G$-action. Consider the affine transformation
\begin{align*}
  \alpha: 
  (G \times \frakg) \times \frakg
  &\longrightarrow \frakg
  \\
  \alpha 
  &\coloneqq 
  +_e \circ
  (\Ad_e \times \id_\frakg) \circ 
  (\pr_1, \pr_3, \pr_2)
\end{align*}
Using the adjoint action, we can equip $G \times \frakg$ with the group structure of the semidirect product
\begin{align*}
  (G \times \frakg) \times (G \times \frakg) 
  &\xrightarrow{~m_{G \ltimes \frakg}~} G \times \frakg
  \\
  * 
  &\xrightarrow{\:~ e_{G \ltimes \frakg} ~\:} 
  G \times \frakg
  \,,
\end{align*}
which is defined as
\begin{align*}
  m_{G \ltimes \frakg}
  &\coloneqq
  (\id_G \times +_e) \circ
  (m \times \Ad_e \times \id_\frakg ) \circ
  (\pr_1, \pr_3, \pr_1, \pr_4, \pr_2)
  \\
  e_{G \ltimes \frakg}
  &\coloneqq (e, 0_e)
  \,.
\end{align*}
We denote this group by $G \ltimes \frakg$.

\begin{Proposition}
\label{prop:TGsemidirect}
Let $(G, m, e)$ be a group object in $\calC$. The trivialization of Proposition~\ref{prop:TGtrivial} is an isomorphism
\begin{equation*}
  TG \cong G \ltimes \frakg
\end{equation*}
of groups.
\end{Proposition}
\begin{proof}
Using~\eqref{eq:T2mActB}, we obtain
\begin{equation*}
\begin{split}
  \Tpart{2} m \circ 
  (\pi_G \times \id_{TG}) \circ
  (j_e \times j_e)
  &= 
  \Tpart{2} m \circ 
  (e \times \id_{TG}) \circ \pr_2 \circ
  (j_e \times j_e)
  \\
  &= 
  \pr_2 \circ (j_e \times j_e)
  \,.    
\end{split}
\end{equation*}
Analogously,
\begin{equation*}
  \Tpart{1} m \circ 
  (\id_{TG} \times \pi_G) \circ
  (j_e \times j_e)
  = 
  \pr_1 \circ (j_e \times j_e)
  \,.
\end{equation*}
Using Proposition~\ref{prop:TSumPartials}, we obtain
\begin{equation*}
\begin{split}
  Tm \circ 
  \chi_{G,G}^{-1} \circ 
  (j_e \times j_e)
  &=
  +_G \circ \bigl(
    \Tpart{1} m \circ (\id_G \times \pi_G),
    \Tpart{2} m \circ (\pi_G \times \id_G)
  \bigr) \circ (j_e \times j_e)
  \\
  &= +_G \circ \bigl( 
  \pr_1 \circ (j_e \times j_e),
  \pr_2 \circ (j_e \times j_e) \bigr)
  \\
  &= +_G \circ (j_e \times j_e)
  \\
  &= j_e \circ +_e
  \,,
\end{split}    
\end{equation*}
where we have used the first diagram of~\eqref{eq:FiberRMod1} that defines $+_e$. In terms of a diagram, this means that 
\begin{equation*}
\begin{tikzcd}[column sep=large]
T_e G \times T_e G
\ar[r, "+_e"]
\ar[d, "j_e \times j_e"']
&
T_e G
\ar[d, "j_e"]
\\
TG \times TG 
\ar[r, "Tm \circ \chi_{G,G}^{-1}"']
&
TG
\end{tikzcd}
\end{equation*}
commutes. In other words, the addition of $T_e G = \frakg$ is the tangent map of the group multiplication at the identity. This implies that we have the short exact sequence of groups
\begin{equation*}
  \frakg \longrightarrow TG \longrightarrow G \longrightarrow 0
  \,,
\end{equation*}
which is right split by the trivialization $G \to G \times \frakg \xrightarrow{\cong} TG$.
\end{proof}

\begin{Corollary}
\label{cor:TAproduct}
Let $(A,\hat{+}, \hat{0})$ be an abelian group object in $\calC$. Then $TA$ is isomorphic to the product group $A \times T_{\hat{0}}A$.
\end{Corollary}

\subsection{Tangent-stable modules}

Throughout this section, we assume that $\calC$ is a cartesian tangent category with an $R$-module structure. 

\begin{Definition}
\label{def:ModTangentStable}
An $R$-module $A$ in $\calC$ will be called \textbf{tangent-stable} if there is an isomorphism of $R$-modules
\begin{equation*}
  T_{\hat{0}} A \cong A
  \,,
\end{equation*}
where $\hat{0}: * \to A$ is the zero of the module.
\end{Definition}

\begin{Proposition}
An $R$-module $A$ in $\calC$ is tangent-stable if and only if its tangent bundle has a trivialization
\begin{equation*}
  TA \cong A \times A
  \,.
\end{equation*}
\end{Proposition}
\begin{proof}
The statement follows from Proposition~\ref{prop:TGtrivial} applied to the underlying abelian group of $A$.
\end{proof}

Let $A$ be a tangent-stable $R$-module. Recall from Equation~\eqref{eq:ValGDef} that
\begin{equation*}
  \Val_A: TA \longrightarrow A
\end{equation*}
denotes the projection onto the fiber of $TA \cong A \times A$. 

\begin{Proposition}
\label{prop:TangStabDiffObj}
If an $R$-module $A$ is tangent-stable, then the diagrams
\begin{equation}
\label{eq:DiffObj1}
\begin{tikzcd}[column sep=large]
TA \times_A TA
\ar[r, "+_A"]
\ar[d, "{(\pr_1, \pr_2)}"']
&
TA
\ar[dd, "\Val_A"]
\\
TA \times TA
\ar[d, "\Val_A \times \Val_A"']
&
\\
A \times A
\ar[r, "\Aadd"']
&
A
\end{tikzcd}
\qquad\qquad
\begin{tikzcd}
A
\ar[r, "0_A"]
\ar[d]
&
TA
\ar[d, "\Val_A"]
\\
*
\ar[r, "\Azero"']
&
A
\end{tikzcd}
\end{equation}
\begin{equation}
\label{eq:DiffObj2}
\begin{tikzcd}[column sep=large]
T(A \times A)
\ar[r, "T\Aadd"]
\ar[d, "\chi_{A,A}"']
&
TA
\ar[dd, "\Val_A"]
\\
TA \times TA
\ar[d, "\Val_A \times \Val_A"']
\\
A \times A
\ar[r, "\Aadd"']
&
A
\end{tikzcd}
\qquad\qquad
\begin{tikzcd}
T*
\ar[r, "T\Azero"]
\ar[d]
&
TA
\ar[d, "\Val_A"]
\\
*
\ar[r, "\Azero"']
&
A
\end{tikzcd}
\end{equation}
\begin{equation}
\label{eq:DiffObj3}
\begin{tikzcd}[column sep=large]
TA 
\ar[r, "\lambda_A"]
\ar[d, "\Val_A"']
&
T^2 A
\ar[d, "T\Val_A"]
\\
A
&
TA
\ar[l, "\Val_A"]
\end{tikzcd}
\end{equation}
commute. 
\end{Proposition}

\begin{proof}
It follows from Proposition~\ref{prop:TGtrivial} that the diagrams~\eqref{eq:DiffObj1}
commute. From Corollary~\ref{cor:TAproduct} it follows that the diagrams~\eqref{eq:DiffObj2} commute. From the property that diagram~\eqref{eq:TanFun4} is a pullback in any tangent category, it follows that diagram~\eqref{eq:DiffObj3} commutes.
\end{proof}

\begin{Remark}
In \cite[Def.~3.1]{CockettCruttwell:2018}, a differential object in a tangent category was defined to be a commutative monoid $A$ such that there is an isomorphism $TA \cong A \times A$ and such that the diagrams~\eqref{eq:DiffObj1}, \eqref{eq:DiffObj2}, and \eqref{eq:DiffObj3} commute. In the earlier paper \cite[Def.~4.8]{CockettCruttwell:2014} of the same authors, the second diagram of \eqref{eq:DiffObj2} was shown to be redundant \cite[Lem.~4.9]{CockettCruttwell:2014}. Proposition~\ref{prop:TangStabDiffObj} shows that an abelian group object is a differential object if and only if it is tangent-stable.
\end{Remark}

\begin{Proposition}
\label{prop:frakgTanStable}
The tangent fibers of a group object are tangent-stable.
\end{Proposition}
\begin{proof}
Let $(G,m,e)$ be a group object. It follows from Proposition~\ref{prop:TGtrivial} that
\begin{equation*}
\begin{split}
  T^2 G 
  &\cong T(G \times \frakg) 
  \cong
  TG \times T\frakg
  \\
  &\cong
  G \times \frakg \times \frakg \times T_{0_e} \frakg
  \,.
\end{split}
\end{equation*}
Using this isomorphism, we see that the kernel of the morphism $(\pi_{TG}, T\pi_G): T^2 G \to T_2 G$ is isomorphic to
\begin{equation*}
\begin{tikzcd}[column sep=6em]
G \times T_{0_e}\frakg 
\ar[r, "{(\pr_1, 0_e \circ \Term, 0_e \circ \Term, \pr_2)}"]
\ar[d, "\pr_1"']
&
G \times \frakg \times \frakg \times T_{0_e}\frakg
\ar[d, "\pr_{1,2,3}"]
\\
G
\ar[r, "{(\id_G, 0_e \circ \Term, 0_e \circ \Term)}"']
&
G \times \frakg \times \frakg
\end{tikzcd}
\end{equation*}
where $\Term$ is the terminal morphism. As is the case for any tangent category, the kernel of $(\pi_{TG}, T\pi_G)$ is isomorphic to $TG \cong G \times \frakg$. We conclude that the $R$-module $\frakg$ is isomorphic to $T_{0_e}\frakg$.
\end{proof}

Proposition~\ref{prop:frakgTanStable} implies that we have an isomorphism
\begin{equation*}
  T^k G \cong G \times \frakg^{2^k-1}
  \,.
\end{equation*}
In order to distinguish the factors of $\frakg$ we enumerate the factors in the product $T^k = T^{(k)} \cdots T^{(2)} T^{(1)}$, where $T^{(i)} = T$. Then we can write
\begin{equation*}
\begin{split}    
  T^{(2)} T^{(1)} G 
  &\cong
  T^{(2)} (G \times T^{(1)}_e G)
  \\
  &\cong T^{(2)} G \times T^{(2)} T^{(1)}_e G
  \\
  &\cong (G \times T^{(2)}_e G) \times (T^{(1)}_e G \times T^{(2)}_0 T^{(1)}_e G)
  \\
  &\cong G \times \frakg_{\{2\}} \times \frakg_{\{1\}} \times \frakg_{\{1,2\}}
  \,,
\end{split}
\end{equation*}
where $\frakg_{\{i\}} = T^{(i)}_e G$, $\frakg_{\{i,j\}} = T^{(j)}_0 T^{(i)}_e G$, etc. In this way, we can label the factors of $\frakg$ in the trivialization of $T^k G$ by the subsets of $\{1, \ldots, k\}$. For a tangent-stable module or abelian group $A$ we obtain the isomorphism
\begin{equation}
\label{eq:TkAdiffProduct}
  T^k\! A \cong \prod_{I \subset \{1, \dots, k\}} A_I
  \,,
\end{equation}
where $A_I = A$ for all $I$. In this notation, the action of an element $\sigma \in S_k$ of the symmetric group on $T^k A$ is given by mapping $A_I$ identically to $A_{\sigma(I)}$.

Let $\pr_J: \prod_{I \subset \{1, \dots, k\}} A_I \to A_J$ denote the projection to the factor $A_J$. Then the morphism $\Val_A$ is given by
\begin{equation*}
  \Val_A:
  TA \xrightarrow{~\cong~} 
  A_{\emptyset} \times A_{\{1\}}
  \xrightarrow{~\pr_{\{1\}}~}
  A_{\{1\}} 
  = A
\end{equation*}
onto the values of the tangent vectors. Viewing $TA$ as an abelian group in its own right, $\Val_{TA}$ is given by the commutative diagram
\begin{equation*}
\begin{tikzcd}[column sep=5em]
T^2 A \ar[r, "\Val_{TA}"] 
\ar[d, "\cong"']
&
TA
\ar[d, "\cong"]
\\
A_{\emptyset} \times 
A_{\{1\}} \times
A_{\{2\}} \times 
A_{\{1,2\}}
\ar[r, "{(\pr_{\{2\}}, \pr_{\{1,2\}})}"']
&
A_{\{2\}} \times 
A_{\{1,2\}}
\end{tikzcd}
\end{equation*}
The tangent morphism of $\Val_A$ is given by
\begin{equation*}
\begin{tikzcd}[column sep=5em]
T^2 A \ar[r, "T\Val_A"] 
\ar[d, "\cong"']
&
TA
\ar[d, "\cong"]
\\
A_{\emptyset} \times 
A_{\{1\}} \times
A_{\{2\}} \times 
A_{\{1,2\}}
\ar[r, "{(\pr_{\{1\}}, \pr_{\{1,2\}})}"']
&
A_{\{1\}} \times 
A_{\{1,2\}}
\end{tikzcd}
\end{equation*}
From the last two commutative diagrams we deduce the following:

\begin{Lemma}
For a tangent-stable module $A$, the following relations hold:
\begin{align}
\label{eq:TEtaEtaTtau}
  T\Val_{A} 
  &= \Val_{TA} \circ \tau_A
  \\
\label{eq:EtaTEtalambda}
  \Val_A
  &=
  \Val_A \circ T\Val_{A} \circ \lambda_{A} 
  \\
\label{eq:EtaTEtalambda2}
  \Val_A \circ \pr_2
  &=
  \Val_A \circ T\Val_{A} \circ (\lambda_2)_A 
  \\
\label{eq:EtaEtaT}
  \Val_A \circ T\Val_{A}
  &= \Val_A \circ \Val_{TA}
\\
\label{eq:EtaTEtaSymmetric01}
  \Val_A \circ T\Val_{A} \circ \tau_A 
  &= \Val_A \circ T\Val_{A}
\end{align}
\end{Lemma}

\begin{Lemma}
\label{lem:AaddTadd}
Let $A$ be a tangent-stable module with addition $\Aadd$. Then
\begin{equation*}
  \Val_A \circ T\Val_A \circ +_{T\!A}
  = \Aadd \circ 
  ( \Val_A \circ T\Val_A \circ \pr_1, 
    \Val_A \circ T\Val_A \circ \pr_2 )
  \,,
\end{equation*}
where $\pr_1, \pr_2: T^2\! A \times_{TA} T^2\! A \to T^2\!A$ are the projections.
\end{Lemma}
\begin{proof}
Let $j_A: TA \times_A TA \to TA \times TA$ denote the natural morphism. Due to its naturality, the diagram
\begin{equation*}
\begin{tikzcd}
T(TA \times_A^{\pi_A, \pi_A} TA)
\ar[r, "(\nu_2)_A"]
\ar[d, "{Tj_A}"]
&
T^2 A \times_{TA}^{T\pi_A, T\pi_A} T^2 A
\ar[r, "\tau_A \times_{T\!A} \tau_A"]
\ar[d, "{j_{TA}}"]
&[1em]
T^2 A \times_{TA}^{\pi_{TA}, \pi_{TA}} T^2 A
\ar[d, "{j_{TA}}"]
\\
T(TA \times TA)
\ar[r, "\chi_{T\!A, T\!A}"']
&
T^2 A \times T^2 A
\ar[r, "\tau_A \times \tau_A"']
&
T^2 A \times T^2 A
\end{tikzcd}
\end{equation*}
commutes. The commutativity is equivalent to the relation
\begin{equation}
\label{eq:AaddTadd01}
  Tj_A \circ
  (\nu_2)_A^{-1} \circ 
  (\tau_A \times_{T\!A} \tau_A)
  =
  \chi_{T\!A,T\!A}^{-1} \circ
  (\tau_A \times \tau_A) \circ
  j_{TA}
  \,,
\end{equation}
where we have used that $\tau_A^2 = \id_A$. Applying the tangent functor to the left diagram of~\eqref{eq:DiffObj1}, we obtain
\begin{equation*}
\begin{split}
  \Val_A \circ T\Val_A \circ +_{TA}
  &=
  \Val_A \circ T\Val_A \circ \tau_A \circ T\!+_A \circ\,
  (\nu_2)_A^{-1} \circ 
  (\tau_A \times_{T\!A} \tau_A)
  \\  
  &=
  \Val_A \circ T\Val_A \circ T\!+_A \circ\,
  (\nu_2)_A^{-1} \circ 
  (\tau_A \times_{T\!A} \tau_A)
  \\  
  &=
  \Val_A \circ 
  T\Aadd \circ 
  T(\Val_A \times \Val_A) \circ 
  Tj_A \circ\,
  (\nu_2)_A^{-1} \circ 
  (\tau_A \times_{T\!A} \tau_A)
  \\  
  &=
  \Aadd \circ 
  (\Val_A \times \Val_A) \circ
  \chi_{A,A} \circ 
  T(\Val_A \times \Val_A) 
  \\
  &{\quad}
  \circ 
  \chi_{T\!A,T\!A}^{-1} \circ
  (\tau_A \times \tau_A) \circ
  j_{TA}
  \\  
  &=
  \Aadd \circ 
  (\Val_A \times \Val_A) \circ
  (T\Val_A \times T\Val_A) \circ 
  (\tau_A \times \tau_A) \circ
  j_{TA}
  \\  
  &=
  \Aadd \circ 
  ( \Val_A \circ T\Val_A \circ \pr_1, 
    \Val_A \circ T\Val_A \circ \pr_2 )
  \,,
\end{split}
\end{equation*}
where we have used the right diagram of~\eqref{eq:TanFun2}, Equation~\eqref{eq:EtaTEtaSymmetric01}, the diagram we obtain by applying the tangent functor to the left diagram of~\eqref{eq:DiffObj1}, the left diagram of~\eqref{eq:DiffObj2} and Equation~\eqref{eq:AaddTadd01}, the naturality of $\chi$, and finally~\eqref{eq:EtaTEtaSymmetric01} again.
\end{proof}

\begin{Corollary}
\label{cor:AaddTminus}
Let $A$ be a tangent-stable module with subtraction $\Aminus$. Then
\begin{equation}
\label{eq:AminusTminus}
  \Val_A \circ T\Val_A \circ -_{T\!A}
  = \Aminus \circ 
  ( \Val_A \circ T\Val_A \circ \pr_1, 
    \Val_A \circ T\Val_A \circ \pr_2 )
  \,,
\end{equation}
where $\pr_1, \pr_2: T^2\! A \times_{TA} T^2\! A \to T^2\!A$ are the projections.
\end{Corollary}

\subsection{The axioms of scalar multiplication}
\label{sec:ScMult}

In order to obtain a Cartan calculus on every object of a tangent category, we have to require certain compatibility conditions of the $R$-module structure with the tangent structure.

\begin{Definition}
A commutative ring object in a cartesian tangent category will be called \textbf{tangent-stable} if it is tangent-stable (Definition~\ref{def:ModTangentStable}) as a module over itself.
\end{Definition}

\begin{Definition}
\label{def:RScalar}
Let $R$ be a commutative ring in a cartesian tangent category $\calC$. An $R$-module structure $\kappa_X: R \times TX \to TX$ on the tangent bundle $\pi:T \to \Id$ will be called a \textbf{scalar multiplication} if $R$ is tangent-stable and if the following diagrams commute for all $X \in \calC$:
\begin{equation}
\label{diag:ScalarMult1}
\begin{tikzcd}[column sep={large}]
R \times TX 
\ar[r, "\id_R \times \lambda_X"] 
\ar[d, "\kappa_X"'] &
R \times T^2 X \ar[d, "\kappa_{TX}"]
\\
TX 
\ar[r, "\lambda_X"'] 
&
T^2 X
\end{tikzcd}
\end{equation}
\begin{equation}
\label{diag:ScalarMult2}
\begin{tikzcd}[column sep=6em]
TR \times TX
\ar[r, "\Tpart{1}\kappa_X"]
\ar[d, "{(\pi_R, \Val_R) \, \times \, \id_{TX}}"']
&[4em]
T^2X
\\
R \times R \times TX
\ar[r, "{\big(\kappa_X \circ (\pr_1,\pr_3), \, 
\kappa_X \circ (\pr_2,\pr_3)\big)}"']
&
T_2 X
\ar[u, "(\lambda_2)_X"']
\end{tikzcd}
\end{equation}
\begin{equation}
\label{diag:ScalarMult3}
\begin{tikzcd}[column sep=3em]
R \times T^2 X
\ar[r, "\Tpart{2}\kappa_X"]
\ar[d, "\id_R \times \tau_X"']
&
T^2 X
\\
R \times T^2 X
\ar[r, "\kappa_{TX}"']
&
T^2 X
\ar[u, "\tau_X"']
\end{tikzcd}
\end{equation}

\end{Definition}


\begin{Proposition}
\label{prop:RmodLinearization}
Let $A$ be an $R$-module in a cartesian tangent category with scalar $R$-multiplication; let $\Amod: R \times A \to A$ denote the module structure. If $A$ is tangent-stable, then the diagram
\begin{equation*}
\begin{tikzcd}
R \times TA
\ar[r, "\Tpart{2} \Amod"]
\ar[d, "\id_R \times \Val_A"']
&
TA
\ar[d, "\Val_A"]
\\
R \times A
\ar[r, "\Amod"']
&
A
\end{tikzcd}
\end{equation*}
commutes.
\end{Proposition}
\begin{proof}
Since $A$ is tangent-stable, the diagram
\begin{equation*}
\begin{tikzcd}
R \times TA
\ar[r, "\kappa_A"]
\ar[d, "\id_R \times \Val_A"']
&
TA
\ar[d, "\Val_A"]
\\
R \times A
\ar[r, "\Amod"']
&
A
\end{tikzcd}
\end{equation*}
commutes. By the naturality of $\kappa$, the diagram
\begin{equation*}
\begin{tikzcd}
R \times T^2 A
\ar[r, "\kappa_{TA}"]
\ar[d, "\id_R \times T\Val_A"']
&
T^2 A
\ar[d, "T \Val_A"]
\\
R \times TA
\ar[r, "\kappa_A"']
&
TA
\end{tikzcd}
\end{equation*}
commutes. Combining the last two diagrams and using Equation~\eqref{eq:EtaEtaT}, we obtain the commutative diagram
\begin{equation}
\label{eq:T2mA04}
\begin{tikzcd}
R \times T^2 A
\ar[r, "\kappa_{TA}"]
\ar[d, "{\id_R \times (\Val_A \circ \Val_{TA})}"']
&
T^2 A
\ar[d, "\Val_A \circ \Val_{TA}"]
\\
R \times A
\ar[r, "\Amod"']
&
A
\end{tikzcd}
\end{equation}
By the functoriality~\eqref{eq:T2functorial} of the partial tangent functor $\Tpart{2}$ we obtain the commutative diagram
\begin{equation}
\label{eq:T2mA02}
\begin{tikzcd}
R \times T^2 A
\ar[r, "\Tpart{2}\kappa_A"]
\ar[d, "\id_R \times T\Val_A"']
&
T^2 A
\ar[d, "T\Val_A"]
\\
R \times TA
\ar[r, "\Tpart{2} \Amod"']
&
TA
\end{tikzcd}
\end{equation}
 We now have
\begin{equation*}
\begin{split}
  \Val_A \circ \Tpart{2} \Amod \circ (\id_R \times T\Val_A)
  &=
  \Val_A \circ T\Val_A \circ \Tpart{2} \kappa_A
  \\
  &= \Val_A \circ T\Val_A \circ \tau_A \circ 
  \kappa_{TA} \circ (\id_R \times \tau_A)
  \\
  &= \Val_A \circ \Val_{TA} \circ 
  \kappa_{TA} \circ (\id_R \times \tau_A)
  \\
  &= \Amod \circ 
  \bigl(\id_R \times (\Val_A \circ \Val_{TA} \circ \tau_A) \bigr)
  \\
  &= \Amod \circ 
  \bigl(\id_R \times (\Val_A \circ T\Val_{TA})\bigr)
  \\
  &= \Amod \circ (\id_R \times \Val_A) \circ (\id_R \times T\Val_{TA})
  \,,
\end{split} 
\end{equation*}
where we have used Diagram~\eqref{eq:T2mA02}, Diagram~\eqref{diag:ScalarMult3}, Equation~\eqref{eq:TEtaEtaTtau}, Diagram~\eqref{eq:T2mA04}, and finally again Equation~\eqref{eq:TEtaEtaTtau}. Since $\Val_A$ is a split epimorphism, so is $T\Val_A$. It follows that $\id_R \times T\Val_A$ is an epimorphism, so that it can be cancelled on both sides of the equation. We conclude that
\begin{equation*}
  \Val_A \circ \Tpart{2} \Amod = \Amod \circ (\id_R \times \Val_A)
  \,,
\end{equation*}
which finishes the proof.
\end{proof}

In other words, Proposition~\ref{prop:RmodLinearization} states that the tangent map of the $R$-multiplication of a tangent-stable $R$-module is the $R$-multiplication itself. 

\begin{Proposition}
\label{prop:LeibnizRuleR}
In a cartesian tangent category with scalar $R$-multiplication, the multiplication $\Rmult$ and the addition $\Radd$ of $R$ satisfy
\begin{equation*}
  \Val_R \circ T\Rmult
  = 
  \Radd \circ \bigl( \Rmult \circ (\Val_R \times \pi_R), 
  \Rmult \circ (\pi_R \times \Val_R) \bigr)
  \circ \chi_{R,R} 
  \,.
\end{equation*}
\end{Proposition}
\begin{proof}
From Proposition~\ref{prop:RmodLinearization} we obtain the relation
\begin{equation}
\label{eq:LeibnizRuleR1}
  \Val_R \circ \Tpart{2} \Rmult = \Rmult \circ (\id_R \times \Val_R)
  \,.
\end{equation}
Since $\Rmult$ is commutative, we also have
\begin{equation}
\label{eq:LeibnizRuleR2}
  \Val_R \circ \Tpart{1} \Rmult = \Rmult \circ (\Val_R \times \id_R)
  \,.
\end{equation}
From Proposition~\ref{prop:TSumPartials} we deduce
\begin{equation*}
\begin{split}
  \Val_R \circ T\Rmult
  &= \Val_R \circ +_R \circ \bigl(
     \Tpart{1} \Rmult \circ (\id_{TR} \times \pi_R), 
     \Tpart{2} \Rmult \circ (\pi_R \times \id_{TR}) \bigr)
     \circ \chi_{R,R}
  \\
  &= \Radd \circ (\Val_R \times \Val_R) \circ \bigl(
     \Tpart{1} \Rmult \circ (\id_{TR} \times \pi_R), 
     \Tpart{2} \Rmult \circ (\pi_R \times \id_{TR}) \bigr)
     \circ \chi_{R,R}
  \\
  &= \Radd \circ \bigl(
     \Val_R \circ \Tpart{1} \Rmult \circ (\id_{TR} \times \pi_R), 
     \Val_R \circ \Tpart{2} \Rmult \circ (\pi_R \times \id_{TR}) \bigr)
     \circ \chi_{R,R}
  \\
  &= \Radd \circ \bigl(
     \Rmult \circ (\Val_R \times \id_R) \circ (\id_{TR} \times \pi_R),
  \\
  &{}\qquad\qquad
     \Rmult \circ (\id_R \times \Val_R) \circ (\pi_R \times \id_{TR}) \bigr)
     \circ \chi_{R,R}
  \\
  &= \Radd \circ \bigl(
     \Rmult \circ (\Val_R \times \pi_R), 
     \Rmult \circ (\pi_R \times \Val_R)
     \bigr)
     \circ \chi_{R,R}
  \,,
\end{split}
\end{equation*}
where we have used the first diagram in~\eqref{eq:DiffObj1} and Equations~\eqref{eq:LeibnizRuleR1} and \eqref{eq:LeibnizRuleR2}.
\end{proof}

\section{Vector fields and scalar functions}
\label{sec:VecActFunc}


\subsection{Naturality of the Lie bracket}

\begin{Definition}
\label{def:fRel}
Let $\phi:X \to Y$ be a morphism in a tangent category. A vector field $v$ on $X$ is called \textbf{$\phi$-related} to a vector field $v'$ on $Y$ if the diagram
\begin{equation}
\label{diag:ProjVect}
\begin{tikzcd}
X \ar[r, "v"] \ar[d, "\phi"']
&
TX \ar[d, "T\phi"]
\\
Y \ar[r, "v'"'] 
&
TY
\end{tikzcd}
\end{equation}
commutes. 
\end{Definition}

\begin{Proposition}
\label{prop:ProjectVectorFields}
Let $\phi: X \to Y$ be a morphism in a tangent category. If two vector fields $v$ and $w$ on $X$ are $\phi$-related to vector fields $v'$ and $w'$ on $Y$, respectively, then $[v,w]$ is $\phi$-related to $[v', w']$.
\end{Proposition}

\begin{proof}
Consider the following diagram:
\begin{equation}
\label{diag:Proj1}
\begin{tikzcd}
X
\ar[r, "{(v, w)}"]
\ar[d, "\phi"']
&
TX \times TX 
\ar[r, "Tw \times Tv"]
\ar[d, "T\phi \times T\phi"]
&[1em]
T^2 X \times T^2 X
\ar[r, "{\id_{T^2X} \times \tau_X}"]
\ar[d, "{T^2\phi \times T^2\phi}"]
&[2em]
T^2 X \times T^2 X
\ar[d, "{T^2\phi \times T^2\phi}"]
\\
Y 
\ar[r, "{(v',w')}"']
&
TY \times TY 
\ar[r, "Tw' \times Tv'"']
&
T^2 Y \times T^2 Y
\ar[r, "{\id_{T^2Y} \times \tau_Y}"']
&
T^2 Y \times T^2 Y
\end{tikzcd}
\end{equation}
The left square and the center square commute by~\eqref{diag:ProjVect} and the functoriality of $T$. The right square commutes by the naturality of $\tau$. It follows that the outer rectangle commutes.

Now, consider the following diagram:
\begin{equation}
\label{diag:Proj2}
\begin{tikzcd}[column sep=2.5em, row sep=2.5ex]
X 
\ar[rrr, "{(Tw \, \circ \, v, \, \tau_X \circ Tv \, \circ \, w)}"]
\ar[ddd, "\phi"'] 
\ar[dr, "\id_X"']
& & &[-4em]
T^2X \times_{TX} T^2X 
\ar[ddd, "T^2\phi \times_{T\phi} T^2\phi"] 
\ar[dl, "i"]
\\
&
X 
\ar[d, "\phi"'] 
\ar[r] 
&
T^2X \times T^2X \ar[d, "T^2\phi \times T^2\phi"]
&
\\
&
Y \ar[r] &
T^2Y \times T^2Y
&
\\
Y \ar[rrr, "{(Tw' \circ \, v', \, \tau_Y \circ Tv' \circ \, w')}"']
\ar[ur, "\id_Y"] 
& & &
T^2Y \times_{TY} T^2Y
\ar[ul, "{i'}"']
\end{tikzcd}    
\end{equation}
The inner square is the commutative outer rectangle of Diagram~\eqref{diag:Proj1}. The left trapezoid commutes trivially. The top and bottom trapezoids commute by the universal property of pullbacks and products. The right trapezoid commutes by the naturality of the inclusions $i$ and $i'$. The map $i'$ is a monomorphism since it is the pullback of the diagonal monomorphism $\Delta: TY \to TY \times TY$ by $\pi_Y \times \pi_Y$. It follows from Lemma~\ref{lem:InnerOuterSquares} that the outer square of Diagram~\eqref{diag:Proj2} commutes. 

Composing the horizontal morphisms of the outer commutative square of Di\-a\-gram~\eqref{diag:Proj2} with the differences $-_{TX}$ and $-_{TY}$, we get the following diagram:
\begin{equation}
\label{diag:Proj3}
\begin{tikzcd}
X 
\ar[r, "{(Tw \, \circ \, v, \, \tau_X \circ Tv \, \circ \, w)}"] 
\ar[d, "\phi"']
&[5em]
T^2X \times_{TX} T^2X 
\ar[r, "-_{TX}"]
\ar[d, "T^2\phi \times_{T\phi} T^2\phi"]
&
T^2X 
\ar[d, "T^2\phi"]
\\
Y 
\ar[r, "{(Tw' \circ \, v', \, \tau_Y \circ Tv' \circ \, w')}"'] 
&
T^2Y \times_{TY} T^2Y 
\ar[r, "-_{TY}"']
&
T^2Y
\end{tikzcd}
\end{equation}
The right square commutes by the naturality of $-$. Thus, the outer rectangle of Diagram~\eqref{diag:Proj3} commutes. The composition of the top horizontal morphisms is $\delta(v,w)$ and that of the bottom horizontal morphisms is $\delta(v',w')$.

Now, consider the following diagram:
\begin{equation}
\label{diag:Proj4}
\begin{tikzcd}[column sep=2.5em, row sep=2.5ex]
X 
\ar[rrr, "{(v,[v,w])}"] 
\ar[ddd, "\phi"'] 
\ar[dr, "\id_X"']
& & &[-2em]
TX \times_X TX 
\ar[ddd, "T\phi \times_\phi T\phi"] 
\ar[dl, "(\lambda_2)_X"]
\\
&
X 
\ar[d, "\phi"'] 
\ar[r, "{\delta(v,w)}"] 
&
T^2X \ar[d, "T^2\phi"]
&
\\
&
Y \ar[r, "{\delta(v',w')}"'] &
T^2Y
&
\\
Y \ar[rrr, "{(w',[v',w'])}"']
\ar[ur, "\id_Y"] 
& & &
TY \times_Y TY
\ar[ul, "(\lambda_2)_Y"']
\end{tikzcd}    
\end{equation}
The inner square is the commutative outer rectangle of~\eqref{diag:Proj3}. The right trapezoid commutes by the naturality of $\lambda_2$. The commutativity of the upper and lower trapezoids follows from~\eqref{eq:deltaBracketRel}. Since monomorphisms are stable under pullbacks and the zero section is a monomorphism, it follows from Diagram~\eqref{diag:lambda2pullback} that the map $(\lambda_2)_Y$ is a monomorphism. Thus, the outer square of Diagram~\eqref{diag:Proj4} commutes by Lemma~\ref{lem:InnerOuterSquares}. 

Lastly, by projecting to the second factor, we get that $T\phi \circ [v,w] = [v',w'] \circ \phi$. We conclude that the bracket $[v,w]$ is $\phi$-related to $[v',w']$.
\end{proof}

\subsection{Action of vector fields on functions}

Let $R$ be a ring object in $\calC$ with addition $\Radd$, zero $\Rzero$, multiplication $\Rmult$, and unit $\Runit$. Recall from Example~\ref{ex:PointsOfR} that the set $\Hom(X,R)$ of $R$-valued morphisms is equipped with a ring structure given by
\begin{equation*}
  f + g \coloneqq \Radd \circ (f,g)
  \,,\qquad
  fg \coloneqq \Rmult \circ (f,g)
\end{equation*}
for all morphisms $f, g \in \Hom(X,R)$. 
Assume that $R$ is tangent-stable. Then we can define an action of vector fields on functions by
\begin{equation}
\label{eq:VecFieldFunc}
  v \cdot f: X \xrightarrow{~v~} TX \xrightarrow{~Tf~} TR \xrightarrow{~\Val_R~} R 
\end{equation}
for all $v \in \Gamma(X,TX)$ and $f \in \Hom(X,R)$.

\begin{Proposition}
\label{prop:VecActFunc}
In a cartesian tangent category with scalar $R$-multiplication, the action~\eqref{eq:VecFieldFunc} is a representation of the Lie algebra of vector fields by derivations on the ring of $R$-valued functions. That is,
\begin{subequations}
\begin{align}
\label{eq:VecActFuncA}
  v \cdot (f + g)
  &= v \cdot f + v \cdot g
  \\
\label{eq:VecActFuncB}
  (v+w) \cdot f
  &=
  v \cdot f + w \cdot f
  \\
\label{eq:VecActFunc1}
  [v,w] \cdot f 
  &= v \cdot (w \cdot f) - w \cdot (v \cdot f)
  \\
\label{eq:VecActFunc2}
  v \cdot (fg) 
  &= (v \cdot f)\,g + f\,(v \cdot g)
  \,,
\end{align}    
\end{subequations}
for all $v, w \in \Gamma(X,TX)$ and $f,g \in \Hom(X,R)$.
\end{Proposition}

\begin{proof}
By a direct calculation, we obtain
\begin{equation*}
\begin{split}
  v \cdot (f+g)
  &=
  \Val_R \circ T\bigl(\Radd \circ (f,g) \bigr) \circ v
  \\
  &=
  \Val_R \circ T\Radd \circ T(f,g) \circ v
  \\
  &=
  \Radd \circ (\Val_R \times \Val_R) \circ 
  \chi_{R,R} \circ T(f,g) \circ v
  \\
  &=
  \Radd \circ (\Val_R \times \Val_R) \circ (Tf,Tg) \circ v
  \\
  &=
  \Radd \circ (\Val_R \circ Tf \circ v, 
               \Val_R \circ Tg \circ v)
  \\
  &=
  v \cdot f + v \cdot g
  \,,
\end{split}
\end{equation*}
where we have used the functoriality of $T$ and the left diagram of~\eqref{eq:DiffObj2}. This shows that Equation~\eqref{eq:VecActFuncA} is satisfied. Similarly,
\begin{equation*}
\begin{split}
  (v + w) \cdot f
  &=
  \Val_R \circ Tf \circ +_X \circ (v,w)
  \\  
  &=
  \Val_R \circ +_R \circ (Tf \times_f Tf) \circ (v,w)
  \\  
  &=
  \Radd \circ (\Val_R \times \Val_R) \circ 
  (Tf \circ v, Tf \circ w)
  \\
  &=
  \Radd \circ (\Val_R \circ Tf \circ v, 
               \Val_R \circ Tf \circ w)
  \\
  &=
  v \cdot f + w \cdot f
  \,,
\end{split}
\end{equation*}
where we have used the naturality of $+_X$ and the left diagram of~\eqref{eq:DiffObj1}. This shows that Equation~\eqref{eq:VecActFuncB} is satisfied.

We now turn to the compatibility of the action with the Lie bracket. For the right side of Equation~\eqref{eq:VecActFunc1}, we express the successive action of two vector fields on a function as
\begin{equation}
\label{eq:VecActFunc01}
\begin{split}
  v \cdot (w \cdot f) 
  &=
  v \cdot ( \Val_R \circ Tf \circ w )
  \\
  &=
  \Val_R \circ T( \Val_R \circ Tf \circ w) \circ v
  \\
  &=
  \Val_R \circ T\Val_R \circ T^2 f \circ Tw \circ v
  \,.
\end{split}
\end{equation}
For the left side of Equation~\eqref{eq:VecActFunc1}, we first show that
\begin{equation*}
\begin{split}
  [v, w] \cdot f
  &=
  \Val_R \circ Tf \circ [v,w]
  \\
  &=
  \Val_R \circ \pr_2 \circ T_2 f \circ (w, [v,w])
  \\
  &=
  \Val_R \circ T\Val_R \circ (\lambda_2)_X \circ 
  T_2 f \circ (w, [v,w])
  \\
  &=
  \Val_R \circ T\Val_R \circ T^2 f \circ 
  (\lambda_2)_X \circ (w, [v,w])
  \\
  &=
  \Val_R \circ T\Val_R \circ T^2 f \circ 
  \delta(v,w)
  \,,
\end{split}
\end{equation*}
where we have used Equation~\eqref{eq:EtaTEtalambda2}, the naturality of $\lambda_2$, and Equation~\eqref{eq:deltaBracketRel}. Starting from this relation, we obtain
\begin{equation*}
\begin{split}
  [v,w] \cdot f
  &=
  \Val_R \circ T\Val_R \circ T^2 f \circ \delta(v,w)
  \\
  &=
  \Val_R \circ T\Val_R \circ T^2 f
  \circ -_{TX} \circ 
  (Tw \circ v, \tau_X \circ Tv \circ w)
  \\
  &=
  \Val_R \circ T\Val_R \circ -_{TR} \circ 
  T_2 T f \circ 
  (Tw \circ v, \tau_X \circ Tv \circ w)
  \\
  &=
  \Val_R \circ T\Val_R \circ -_{TR} \circ 
  (T^2f \circ Tw \circ v, T^2f \circ \tau_X \circ Tv \circ w)
  \\
  &=
  \Rminus \circ 
  (\Val_R \circ T\Val_R \circ T^2 f \circ Tw \circ v,
   \Val_R \circ T\Val_R \circ \tau_R \circ T^2 f \circ Tv \circ w)
  \\
  &=
  \Rminus \circ 
  (\Val_R \circ T\Val_R \circ T^2 f \circ Tw \circ v,
   \Val_R \circ T\Val_R \circ T^2 f \circ Tv \circ w)
  \\
  &=
  v \cdot (w \cdot f) - w \cdot (v \cdot f)
  \,,
\end{split}
\end{equation*}
where we have used the defining Equation~\eqref{eq:DeltaOrig} of $\delta(v,w)$, the naturality of the subtraction $-$, the naturality of $\tau$ and Equation~\eqref{eq:AminusTminus}, Equation~\eqref{eq:EtaTEtaSymmetric01}, and finally~\eqref{eq:VecActFunc01}. We conclude that Equation~\eqref{eq:VecActFunc1} holds.

The derivation property can be checked by a straightforward calculation,
\begin{equation*}
\begin{split}
  v \cdot (fg)
  &= \Val_R \circ T\bigl( \Rmult \circ (f,g) \big) \circ v
  \\
  &= \Val_R \circ T\Rmult \circ T(f,g) \circ v
  \\
  &= \Radd \circ 
  \bigl(\Rmult \circ (\Val_R \times \pi_R), 
        \Rmult \circ (\pi_R \times \Val_R) \bigr) \circ
  \chi_{R,R} \circ T(f,g) \circ v
  \\
  &= \Radd \circ 
  \bigl(\Rmult \circ (\Val_R \times \pi_R), 
        \Rmult \circ (\pi_R \times \Val_R) \bigr) \circ
  (Tf, Tg) \circ v
  \\
  &= \Radd \circ 
  \bigl(\Rmult \circ (\Val_R \circ Tf \circ v, g), 
        \Rmult \circ (f, \Val_R \circ Tg \circ v) \bigr)
  \\
  &= \Radd \circ
  \bigl( (v \cdot f) g, f(v \cdot g) \bigr)
  \\
  &= (v \cdot f) g + f(v \cdot g)
  \,,
\end{split}
\end{equation*}
where in the third step we have used \eqref{eq:LeibnizRuleR1} and in the fifth step that 
\begin{equation*}
  \pi_R \circ Tf \circ v
  = f \circ \pi_X \circ v
  = f
  \,.
\end{equation*}
This finishes the proof.
\end{proof}

\begin{Warning}
Proposition~\ref{prop:VecActFunc} shows that there is a homomorphism from the Lie algebra of vector fields on $X$ to the Lie algebra of derivations of the ring $\Hom(X,R)$, where the Lie bracket is given by the commutator. However, this homomorphism is generally neither injective nor surjective, so that we cannot identify vector fields on $X$ with derivations on the structure ring of $X$.
\end{Warning}

\subsection{Module structure and Leibniz rule}

Assume that the tangent structure has an $R$-module structure $\kappa$, so that $TX \to X$ is an ($X \times R \to X$)-module in $\calC_{/X}$. Applying the functor of sections, we see that $\Gamma(X, TX)$ is equipped with the structure of a module over $\Gamma(X, X \times R) \cong \Hom(X,R)$, given by
\begin{equation}
\label{eq:fvDef}
  fv \coloneqq \kappa_X \circ (f,v)
  \,.
\end{equation}

\begin{Definition}
\label{def:fRelFunc}
Let $\phi:X \to Y$ be a morphism in some category. A morphism $f:X \to R$ will be called \textbf{$\phi$-related} to a morphism $f':Y \to R$ if the diagram
\begin{equation}
\label{diag:ProjFunc}
\begin{tikzcd}
X \ar[r, "f"] \ar[d, "\phi"']
&
R
\\
Y \ar[ru, "f'"'] 
&
\end{tikzcd}
\end{equation}
commutes. 
\end{Definition}

\begin{Proposition}
\label{prop:ProjectFunctions}
Let $\phi: X \to Y$ be a morphism in a cartesian tangent category with scalar $R$-multiplication. If $v \in \Gamma(X,TX)$ is $\phi$-related to $v' \in \Gamma(Y,TY)$ and $f:X \to R$ is $\phi$-related to $f': Y \to R$, then:
\begin{itemize}

\item[(i)] $fv$ is $\phi$-related to $f'v'$;

\item[(ii)] $v \cdot f$ is $\phi$-related to $v' \cdot f'$.

\end{itemize}
\end{Proposition}

\begin{proof}
For (i) we consider the following diagram:
\begin{equation*}
\begin{tikzcd}
X 
\ar[r, "{(f,v)}"] 
\ar[d, "\phi"']
&
R \times TX 
\ar[r, "\kappa_X"]
\ar[d, "\id_R \times T\phi"']
&
TX 
\ar[d, "T\phi"]
\\
Y 
\ar[r, "{(f',v')}"'] 
&
R \times TY 
\ar[r, "\kappa_Y"']
& 
TY
\end{tikzcd}
\end{equation*}
where $\kappa$ denotes the scalar $R$-multiplication. The square on the left commutes since $f$ and $v$ are $\phi$-related to $f'$ and $v'$, respectively. The square on the right commutes by the naturality of the scalar multiplication. The commutativity of the outer rectangle tells us that $fv$ is $\phi$-related to $f'v'$.

For (ii) we consider the following diagram:
\begin{equation*}
\begin{tikzcd}
X 
\ar[r, "v"] 
\ar[d, "\phi"']
&
TX 
\ar[r, "Tf"]
\ar[d, "T\phi"']
&
TR 
\ar[r, "\Val_R"]
\ar[d, "\id"]
& R
\ar[d, "\id"]
\\
Y 
\ar[r, "v'"'] 
&
TY 
\ar[r, "Tf'"']
& 
TR
\ar[r, "\Val_R"']
&
R
\end{tikzcd}
\end{equation*}
The square on the left commutes since $v$ is $\phi$-related to $v'$. The square in the middle commutes since $f$ is $\phi$-related to $f'$ and since $T$ is a functor. The square on the right commutes trivially. Thus, the outer rectangle commutes. The composition of the upper horizontal arrows is $v \cdot f$ and that of the lower horizontal arrows is $v' \cdot f'$. We conclude that $v \cdot f$ is $\phi$-related to $v' \cdot f'$.
\end{proof}

\begin{Proposition}
\label{prop:LeibnizRule}
In a cartesian tangent category with scalar $R$-multiplication, the \textbf{Leibniz rule}
\begin{equation}
\label{eq:LeibnizRule}
  [v,fw] = (v \cdot f)w + f[v,w] 
\end{equation}
holds for all vector fields $v, w \in \Gamma(X,TX)$ and all functions $f: X \to R$.
\end{Proposition}

\begin{proof}
In~\eqref{eq:DeltaOrig} we have defined
\begin{equation}
\label{eq:LeibRuleVecField01}
  \delta(v, fw)
  = -_{TX} \circ 
  \bigl( T(fw) \circ v, \tau_X \circ Tv \circ fw \bigr)
  \,,
\end{equation}
which is related to the Lie bracket by~\eqref{eq:deltaBracketRel}. Using the axioms of the scalar multiplication, we obtain the following relation
\begin{equation*}
\begin{split}
  T(fw)
  &=
  T\bigl( \kappa_X \circ (f,w) \bigr)
  \\
  &=
  T\kappa_X \circ T(f,w)
  \\
  &=
  +_{TX} \circ \bigl(
  \Tpart{1} \kappa_X \circ (\id_{TR} \times \pi_{TX}),
  \Tpart{2} \kappa_X \circ (\pi_R \times \id_{T^2X})
  \bigr)
  \circ \chi_{R,TX} \circ T(f,w)
  \\
  &=
  +_{TX} \circ \bigl(
  \Tpart{1} \kappa_X \circ (\id_{TR} \times \pi_{TX}),
  \Tpart{2} \kappa_X \circ (\pi_R \times \id_{T^2X})
  \bigr)
  \circ (Tf,Tw)
  \\
  &=
  +_{TX} \circ \bigl(
  \Tpart{1} \kappa_X \circ (Tf, \pi_{TX} \circ Tw),
  \Tpart{2} \kappa_X \circ (\pi_R \circ Tf, Tw)
  \bigr)
  \\
  &=
  +_{TX} \circ \bigl(
  \Tpart{1} \kappa_X \circ (Tf, w \circ \pi_X),
  \Tpart{2} \kappa_X \circ (f \circ \pi_X, Tw)
  \bigr)
  \,,
\end{split}
\end{equation*}
where we have used Proposition~\ref{prop:TSumPartials} and in the last step the naturality of $\pi$. Precomposing with $v$, we obtain the sum
\begin{equation}
\label{eq:LeibRuleVecField02}
  T(fw) \circ v
  =  
  +_{TX} \circ \bigl(
  \Tpart{1} \kappa_X \circ (Tf \circ v, w),
  \Tpart{2} \kappa_X \circ (f, Tw \circ v)
  \bigr)
  \,,
\end{equation}
for the first term in~\eqref{eq:LeibRuleVecField01}. The second term in~\eqref{eq:LeibRuleVecField01} can be rewritten as
\begin{equation}
\label{eq:LeibRuleVecField03}
\begin{split}
  \tau_X \circ Tv \circ fw
  &=
  \tau_X \circ Tv \circ \kappa_X \circ (f,w)
  \\
  &=
  \tau_X \circ \kappa_{TX} \circ 
  (\id_R \times Tv) \circ (f,w)
  \\
  &=
  \tau_X \circ \kappa_{TX} \circ (f, Tv \circ w)
  \\
  &=
  \Tpart{2}\kappa_X \circ (\id_R \times \tau_X) \circ 
  (f,Tv \circ w)
  \\
  &=
  \Tpart{2}\kappa_X \circ (f, \tau_X \circ Tv \circ w)
  \,,
\end{split}
\end{equation}
where we have used the naturality of $\kappa_X$ and~\eqref{diag:ScalarMult3}. By inserting~\eqref{eq:LeibRuleVecField02} and~\eqref{eq:LeibRuleVecField03} into~\eqref{eq:LeibRuleVecField01}, we obtain
\begin{equation}
\label{eq:LeibRuleVecField04}
\begin{split}
  \delta(v, fw)
  &=
  -_{TX} \circ \bigl( 
    +_{TX} \circ \bigl(
  \Tpart{1} \kappa_X \circ (Tf \circ v, w),
  \Tpart{2} \kappa_X \circ (f, Tw \circ v)
  \bigr),
  \\
  &{}\qquad
  \Tpart{2}\kappa_X \circ (f, \tau_X \circ Tv \circ w)  
  \bigr)
  \\
  &=
  +_{TX} \circ \bigl( 
  \Tpart{1} \kappa_X \circ (Tf \circ v, w),
  \\
  &{}\qquad
  -_{TX} \circ \bigl(
  \Tpart{2} \kappa_X \circ (f, Tw \circ v),
  \Tpart{2}\kappa_X \circ (f, \tau_X \circ Tv \circ w)
  \bigr)
  \bigr)
  \,,
\end{split}
\end{equation}
where we have used the associativity of addition and subtraction. The first summand of~\eqref{eq:LeibRuleVecField04} can be expressed as
\begin{equation}
\label{eq:LeibRuleVecField05}
\begin{split}
  &
  \Tpart{1} \kappa_X \circ (Tf \circ v, w)
  \\
  ={}&
  (\lambda_2)_X \circ 
    \bigl( \kappa_X \circ (\pr_1, \pr_3), \kappa_X \circ (\pr_2, \pr_3) \bigr) 
    \circ \bigl( (\pi_R, \Val_R) \times \id_{TX} \bigr) \circ 
    (Tf \circ v, w)
  \\
  ={}&
  (\lambda_2)_X \circ 
    \bigl( \kappa_X \circ (\pr_1, \pr_3), \kappa_X \circ (\pr_2, \pr_3) \bigr) 
    \circ (\pi_R \circ Tf \circ v, \Val_R \circ Tf \circ v, w)
  \\
  ={}&
  (\lambda_2)_X \circ 
  \bigl( 
    \kappa_X \circ (f \circ \pi_X \circ v,w), 
    \kappa_X \circ (v \cdot f, w)
  \bigr)
  \\
  ={}&
  (\lambda_2)_X \circ 
  \bigl( 
    \kappa_X \circ (f,w), 
    \kappa_X \circ (v \cdot f, w)
  \bigr)
  \\
  ={}&
  (\lambda_2)_X \circ \bigl(fw, (v \cdot f)w \bigr)
  \,,
\end{split}
\end{equation}
where we have used Diagram~\eqref{diag:ScalarMult2}, the naturality of $\pi$, Definition~\eqref{eq:VecFieldFunc} of the action of vector fields on functions, and that $v$ is a vector field. For the second term we obtain
\begin{equation}
\label{eq:LeibRuleVecField06}
\begin{split}
  {}&
  -_{TX} \circ \bigl( 
  \Tpart{2}\kappa_X \circ (f, Tw \circ v),   \Tpart{2}\kappa_X \circ (f, \tau_X \circ Tv \circ w)
  \bigr)
  \\
  ={}&
  \Tpart{2}\kappa_X \circ \bigl(f, 
  -_{TX} \circ (Tw \circ v, \tau_X \circ Tv \circ w) \bigr)
  \\
  ={}&
  \Tpart{2}\kappa_X \circ \bigl(f, \delta(v,w) \bigr)
  \\
  ={}&    
  \tau_X \circ \kappa_{TX} \circ \bigl(f,  
    \tau_X \circ (\lambda_2)_X \circ (w, [v,w])
    \bigr)
  \\
  ={}&
  \tau_X \circ \kappa_{TX} \circ \bigl(f,  
    +_{TX} \circ ( T0_X \circ w, \lambda_X \circ [v,w])
    \bigr)
  \\
  ={}&
  \tau_X \circ +_{TX} \circ \bigl(
    \kappa_{TX} \circ (f, T0_X \circ w), 
    \kappa_{TX} \circ (f, \lambda_X \circ [v,w]) \bigr)
  \\
  ={}&
  \tau_X \circ +_{TX} \circ \bigl(
    T0_X \circ \kappa_X \circ (f,w), \lambda_X \circ \kappa_X \circ (f,[v,w]) \bigr)
  \\  
  ={}&
  \tau_X \circ +_{TX} \circ (T0_X \times_{0_X} \lambda_X) \circ (fw,f[v,w])
  \\
  ={}&
  (\lambda_2)_X \circ (fw,f[v,w])
  \,,
\end{split}
\end{equation}
where we have used that the partial tangent morphism is linear~\eqref{eq:T2linear}, Definition~\eqref{eq:DeltaOrig} of $\delta(v,w)$, Diagram~\eqref{diag:ScalarMult3} and Relation~\eqref{eq:deltaBracketRel}, Definition~\eqref{eq:VertLiftExt} of $(\lambda_2)_X$, linearity of $\kappa_{TX}$ in $T^2X$ (as expressed by Diagram~(v) in Proposition~\ref{prop:RModuleBundle}), the naturality of $\kappa$ and Diagram~\eqref{diag:ScalarMult1}, and, lastly, Definition~\eqref{eq:VertLiftExt} of $(\lambda_2)_X$ again. 

By inserting~\eqref{eq:LeibRuleVecField05} and \eqref{eq:LeibRuleVecField06} in \eqref{eq:LeibRuleVecField04}, we obtain
\begin{equation*}
\begin{split}
  (\lambda_2)_X \circ (fw, [v, fw])
  &= \delta(v, fw)
  \\
  &=
  +_{TX} \circ \bigl(
  (\lambda_2)_X \circ (fw, (v \cdot f)w ),
  (\lambda_2)_X \circ (fw,f[v,w])  
  \bigr)
  \\
  &=
  (\lambda_2)_X \circ \bigl(
   (fw, +_{TX} \circ \bigl( (v \cdot f)w ), f[v,w]\bigr)  
  \bigr)
  \\
  &=
  (\lambda_2)_X \circ \bigl(fw, (v \cdot f)w + f[v,w] \bigr)
  \,,
\end{split}   
\end{equation*}
where we have used that $(\lambda_2)_X$ is linear in the second argument, expressed by the commutativity of Diagram~\eqref{diag:lambda2linear}. Since $(\lambda_2)_X$ is a monomorphism, we can cancel it on both sides of the last equation, which yields Equation~\eqref{eq:LeibnizRule}.
\end{proof}

As is the case for any pullback, the map
\begin{equation*}
\begin{aligned}
  \Hom(*,R) 
  &\longrightarrow \Hom(X,R)
  \\
  (* \to R)
  &\longmapsto
  (X \to * \to R)
\end{aligned}
\end{equation*}
is a homomorphism of rings. Its image consists of the constant $R$-valued functions on $X$. By precomposing the $\Hom(X,R)$-module structure of $\Gamma(X,TX)$ with this ring homomorphism, we equip $\Gamma(X,TX)$ with the structure of a $\Hom(*,R)$-module.

\begin{Corollary}
Let $X$ be an object in a cartesian tangent category with scalar $R$-multiplication. Then the Lie bracket on $\Gamma(X,TX)$ is $\Hom(*,R)$-bilinear.
\end{Corollary}
\begin{proof}
Let us denote the terminal morphism by $\Term_X:  X \to *$. Since the tangent structure is cartesian, the tangent functor maps the terminal object to the terminal object, so that the bundle projection and the zero section of $T* = *$ are the identity. This implies that the tangent morphism of a point $r: * \to R$ satisfies $Tr = Tr \circ 0_* = 0_R \circ r$. For the action of a vector field $v$ on a constant function, we obtain
\begin{equation*}
\begin{split}
  v \cdot (r\, \circ\, \Term_X)
  &=
  \Val_R \circ Tr \circ T\Term_X \circ v
  \\
  &=
  \Val_R \circ 0_R \circ r\, \circ\, T\Term_X \circ v
  \\
  &=
  \Rzero\, \circ\, \Term_R \circ r \, \circ\, T\Term_X \circ v
  \\
  &=
  \Rzero\, \circ\, \Term_X
  \\
  &= 0
  \,,
\end{split}
\end{equation*}
where we have used Definition~\eqref{eq:VecFieldFunc} of the action and the right diagram of~\eqref{eq:DiffObj1} for $A = R$. It follows from Proposition~\ref{prop:LeibnizRule} that $[v,fw] = f[v,w]$ for all constant functions $f: X \to * \to R$.
\end{proof}

\section{Cartan calculus}
\label{sec:Cartan}

\subsection{Abstract Lie algebroids and Lie-Rinehart algebras}

\begin{Definition}[Definition~7.1 in \cite{AintablianBlohmann:2025}]
\label{def:AbstractLieAlgd}
Let $\calC$ be a cartesian tangent category with scalar $R$-mul\-ti\-pli\-ca\-tion. An \textbf{abstract Lie algebroid} in $\calC$ consists of a bundle of $R$-modules $A \to X$, a morphism $\rho: A \to TX$ of bundles of $R$-modules, called the \textbf{anchor}, and a Lie bracket on the abelian group $\Gamma(X,A)$, such that
\begin{align}
\label{eq:LieAlgdLeibniz}
  [a, fb] &= f[a,b] + \bigl((\rho \circ a) \cdot f \bigr) b
  \\
\label{eq:anchorLieAlgMap}
  \rho \circ [a,b] &= [\rho \circ a , \rho \circ b]
\end{align}
for all sections $a$, $b$ of $A \to X$ and all morphisms $f: X \to R$ in $\catC$.
\end{Definition}

Lie algebroids are closely related to Lie-Rinehart algebras. This remains true for abstract Lie algebroids. We first recall the notion of Lie-Rinehart algebras introduced in \cite{Rinehart:1963}.

\begin{Definition}
\label{def:LieRinehartAlg}
A \textbf{Lie-Rinehart algebra} over a commutative ring $\calO$ is an $\calO$-module $\calX$ equipped with a Lie bracket and an action of $\calX$ on $\calO$ by derivations, satisfying
\begin{equation}
\label{eq:LieRineLeibniz}
  [v, fw] = (v \cdot f)\, w + f[v,w]
\end{equation}
for all $v, w \in \calX$ and $f \in \calO$.
\end{Definition}

\begin{Remark}
Definition~\ref{def:LieRinehartAlg} generalizes to a commutative algebra $\calO$ over a commutative ring $k$ (typically a field) \cite{Huebschmann:1998}. In that case we have to assume that the Lie bracket is $k$-linear.
\end{Remark}

\begin{Proposition}
\label{prop:LieAlgdRinehart}
Let $A \to X$ be an abstract Lie algebroid in a cartesian tangent category with scalar $R$-multiplication. Let $\calO \coloneqq \Hom(X,R)$ and $\calA \coloneqq \Gamma(X,A)$. Then the $\calO$-module $\calA$ with the action $a \cdot f \coloneqq (\rho \circ a) \cdot f$ on $\calO$ is a Lie-Rinehart algebra.
\end{Proposition}

\subsection{The Cartan calculus of Lie-Rinehart forms}

Let $\Omega$ be a graded commutative algebra. We denote the product of $\alpha$ and $\beta$ by $\alpha \wedge \beta$. The degree-zero component $\Omega^0$ is a commutative algebra that lies in the center of $\Omega$, so $\Omega$ is a graded $\Omega^0$-module. A graded derivation of $\Omega$ of degree $p$ is a linear map $\delta: \Omega \to \Omega[p]$ that satisfies
\begin{equation*}
  \delta(\alpha \wedge \beta)
  = \delta\alpha \wedge \beta + (-1)^p \alpha \wedge \delta\beta
  \,.
\end{equation*}
If $\epsilon$ is a graded derivation of degree $q$, then the graded commutator
\begin{equation*}
  [\delta, \epsilon] 
  = \delta\epsilon - (-1)^{pq} \epsilon\delta
\end{equation*}
is a graded derivation of degree $p+q$.

\begin{Definition}
\label{def:CartanCalc}
A \textbf{Cartan calculus} consists of a differential graded commutative algebra $(\Omega, d)$, an $\Omega^0$-module $\calX$ equipped with a Lie bracket, and for every $v \in \calX$ a graded derivation $\iota_v$ of degree $-1$, such that
\begin{equation}
\label{eq:InnDerOlinear}
  \iota_{v+w} = \iota_v + \iota_w
  \,,\quad
  \iota_{fv}\omega = f \wedge \iota_v\omega
\end{equation}
and, setting $\Lie_v \coloneqq [\iota_v, d]$, the bracket relations
\begin{equation}
\label{eq:CartanCalcBrackets}
\begin{gathered}{}
  [\iota_v, \iota_w] = 0 \,,\quad
  [\Lie_v, \iota_w] = \iota_{[v,w]} \,,\quad
  [\Lie_v, \Lie_w] = \Lie_{[v,w]} \,,\quad
  [\Lie_v, d] = 0
\end{gathered}
\end{equation}
hold for all $v, w \in \calX$, $f \in \Omega^0$, and $\omega \in \Omega$.
\end{Definition}

\begin{Remark}
The differential $d$ satisfies $[d,d] = 2d^2 = 0$. Conversely, if $\Omega$ has no 2-torsion, then the relation $[d,d] = 0$ implies $d^2 = 0$.
\end{Remark}

The derivation $\iota_v$ is called the \textbf{inner derivative} and $\Lie_v$ the \textbf{Lie derivative} with respect to $v$. If $\Omega$ has no 2-torsion, the bracket relations~\eqref{eq:CartanCalcBrackets} of Definition~\ref{def:CartanCalc} are redundant. If we only assume that $[\iota_v, \iota_w] = 0$ and $[\Lie_v, \iota_w] = \iota_{[v,w]}$, it follows that
\begin{equation*}
\begin{split}
  [\Lie_v, d]
  &=
  [[\iota_v, d], d]
  =
  [\iota_v, [d, d]] - [[\iota_v, d], d]
  \\
  &=
  - [\Lie_v, d]
  \,,
\end{split}    
\end{equation*}
where we have used the graded Jacobi identity. Since $\Omega$ is assumed to have no 2-torsion, it follows that $[\Lie_v,d] = 0$. Using this relation, we obtain
\begin{equation*}
\begin{split}
  [\Lie_v, \Lie_w]
  &=
  [\Lie_v, [\iota_w, d] ]
  =
  [[\Lie_v, \iota_w], d] +   [\iota_w, [\Lie_v, d] ]
  =
  [\iota_{[v,w]}, d] 
  \\
  &=
  \Lie_{[v,w]} 
  \,.
\end{split}
\end{equation*}
We see that the first two bracket relations imply the remaining two.

To every Lie-Rinehart algebra $\calX$ over $\calO$, we can associate a Cartan calculus as follows. A \textbf{Lie-Rinehart $n$-form} is a map 
\begin{equation*}
  \omega: \underbrace{
  \calX \times \ldots \times \calX}_{\text{$n$ factors}}
  \longrightarrow \calO
\end{equation*}
that is antisymmetric and $\calO$-linear in every argument. A $0$-form is an element of $\calO$. By the universal property of the algebraic wedge product over $\calO$, we can identify $\omega$ with an $\calO$-linear map
\begin{equation*}
  \omega: \wedge^n_\calO \calX 
  \longrightarrow \calO
  \,.
\end{equation*}
The graded $\calO$-module of forms
\begin{equation}
  \Omega^\bullet(\calX) 
  \coloneqq \Hom_\calO(\wedge^\bullet_\calO \calX, \calO)
\end{equation}
has the structure of a graded commutative $\calO$-algebra
\begin{align*}
   \Omega^p(\calX) \times \Omega^q(\calX)
   &\longrightarrow \Omega^{p+q}(\calX)
   \\
   (\alpha,\beta)
   &\longmapsto
   \alpha \wedge \beta
\end{align*}
defined by the antisymmetrization of the pointwise multiplication of forms,
\begin{equation*}
  (\alpha \wedge \beta)(v_1, \ldots, v_{p+q})
  \coloneqq \tfrac{1}{(p+q)!}\sum_{\mathclap{\sigma \in S_{p+q}}}
    (-1)^{|\sigma|}
    \alpha(v_{\sigma(1)}, \ldots, v_{\sigma(p)}) \,
    \beta(v_{\sigma(p+1)}, \ldots, v_{\sigma(p+q)})
    \,.
\end{equation*}
The \textbf{differential} of $\omega$ is the $(n+1)$-form 
\begin{equation}
\label{eq:differential}
\begin{split}
  (d\omega)(v_0, \ldots, v_n)
  \coloneqq &\sum_{\mathclap{0 \leq i \leq n}} (-1)^i 
    v_i \cdot \omega(v_0, \ldots, \hat{v}_i, \ldots, v_n)
  \\
    + &\sum_{\mathclap{0 \leq i < j \leq n}}
    (-1)^{i+j} \omega([v_i, v_j], \ldots, \hat{v}_i, \ldots, \hat{v}_j, \ldots, v_n)
  \,,    
\end{split}
\end{equation}
given by the usual Chevalley-Eilenberg formula. Standard computations show that $d\omega$ is antisymmetric and $\calO$-linear in every argument and that $d^2 = 0$. By a standard calculation, we can check that the differential is a graded derivation,
\begin{equation*}
  d(\alpha \wedge \beta)
  = d\alpha \wedge \beta + (-1)^p \alpha \wedge d\beta
  \,.
\end{equation*}
We conclude that $\Omega(\calX)$ is a differential graded algebra. The \textbf{inner derivative} is the operator that inserts an element $v \in \calX$ into a form:
\begin{equation}
\label{eq:InnDeriv}
\begin{aligned}
   \iota_v: \Omega^n(\calX)
   &\longrightarrow \Omega^{n-1}(\calX)
   \\
   (\iota_v \omega)(v_1, \ldots, v_{n-1})
   &\coloneqq
   \omega(v, v_1, \ldots, v_{n-1})
   \,.
\end{aligned}
\end{equation}
Since $\omega$ is by definition $\calO$-linear, so is the inner derivative. That is, Equations~\eqref{eq:InnDerOlinear} hold. It follows from the antisymmetry of $\omega$ that
  \begin{equation*}
  \iota_v \iota_w \omega = -\iota_w \iota_v \omega
\end{equation*}
for all $v, w \in \calX$. Moreover, it is straightforward to show that $\iota_v$ is a graded derivation. The main result is the following:

\begin{Proposition}
\label{prop:ACartanCalc}
Let $\calX$ be a Lie-Rinehart algebra over $\calO$. The differential graded commutative algebra $(\Omega(\calX), d)$, the Lie algebra $\calX$, and the inner derivative~\eqref{eq:InnDeriv} constitute a Cartan calculus.
\end{Proposition}

\begin{Example}
Let $M$ be a manifold. The Lie-Rinehart algebra of the tangent algebroid $TM \to M$ is given by the $\bbR$-algebra $\calO = C^\infty(M)$ of smooth functions, the Lie algebra $\calX \coloneqq \calX(M)$ of vector fields with the Lie bracket, and the action of vector fields as derivations on $C^\infty(M)$. From Proposition~\ref{prop:ACartanCalc} we retrieve the usual Cartan calculus given by the graded algebra of differential forms with the de Rham differential, inner derivative, and Lie derivative.
\end{Example}

\begin{Proposition}
\label{prop:CartanCalcLieRine}
Let $(\Omega, \calX)$ be a Cartan calculus such that the map $\iota: \calX \to \Der(\Omega)$ is injective. Then $\calX$ is a Lie-Rinehart algebra over $\calO = \Omega^0$ with action $v \cdot f = \iota_v df$ of $v \in \calX$ on $f \in \calO$.
\end{Proposition}
\begin{proof}
Since $[\Lie_v, \Lie_w] = \Lie_{[v,w]}$, the Lie derivative defines an action of $\calX$ on $\Omega$ by derivations. Since $\Lie_v$ has degree 0, the action restricts to $\calO = \Omega^0$. With $\Lie_v = \iota d + d \iota_v$, we can write the action as $v \cdot f = \Lie_v f = \iota_v df$. From Equation~\eqref{eq:InnDerOlinear} we get
\begin{equation*}
\begin{split}
  \iota_{[v, fw]} \omega
  &= [\Lie_v, \iota_{fw}]\omega
  \\
  &= \Lie_v(\iota_{fw} \omega) - \iota_{fw}(\Lie_v \omega)
  \\
  &= \Lie_v(f \wedge \iota_w \omega) - f \wedge \iota_w(\Lie_v \omega)
  \\
  &= (\Lie_v f) \wedge \iota_w \omega
    + f \wedge \Lie_v(\iota_w \omega)
    - f \wedge \iota_w(\Lie_v \omega)
  \\
  &= (v \cdot f) \wedge \iota_w \omega + f \wedge [\Lie_v, \iota_w]\omega
  \\
  &= \iota_{(v \cdot f)w} \omega + f \wedge \iota_{[v,w]} \omega
  \\
  &= \iota_{(v \cdot f)w + f[v,w]} \omega
  \,.
\end{split}
\end{equation*}
Since $v \mapsto \iota_v$ is injective, it follows that the Leibniz rule~\eqref{eq:LieRineLeibniz} holds. We conclude that $\calX$ is a Lie-Rinehart algebra over $\calO$.    
\end{proof}


\begin{Remark}
We posit that the constructions of Propositions~\ref{prop:ACartanCalc} and~\ref{prop:CartanCalcLieRine} define an equivalence between the category of Lie-Rinehart algebras and the category of Cartan calculi with injective inner-derivative map.
\end{Remark}

\subsection{Cartan calculus in tangent categories}

\begin{Proposition}
\label{prop:CartanOnA}
Let $A \to X$ be an abstract Lie algebroid in a cartesian tangent category with scalar $R$-multiplication. Let $\calO \coloneqq \Hom(X,R)$. Then the Lie algebra of sections $\calA \coloneqq \Gamma(X,A)$ and the graded commutative $\calO$-algebra of forms,
\begin{equation*}
  \Omega^\bullet(\calA) 
  \coloneqq 
  \Hom_{\calO}\bigl(\wedge^\bullet_{\calO} \calA, 
    \calO \bigr)
  \,,
\end{equation*}
together with the Chevalley-Eilenberg differential~\eqref{eq:differential} and the inner derivative~\eqref{eq:InnDeriv} constitute a Cartan calculus in the sense of Definition~\ref{def:CartanCalc}.
\end{Proposition}
\begin{proof}
The statement follows from first applying Proposition~\ref{prop:LieAlgdRinehart} to the abstract Lie algebroid and then Proposition~\ref{prop:ACartanCalc} to the resulting Lie-Rinehart algebra.
\end{proof}

\begin{Theorem}
\label{thm:TangentLieAlgd}
Let $X$ be an object in a cartesian tangent category with scalar $R$-multiplication. Then the tangent bundle $TX \to X$ with the anchor $\id: TX \to TX$ and the Lie bracket of vector fields is an abstract Lie algebroid, called the \textbf{tangent Lie algebroid} of $X$. 
\end{Theorem}
\begin{proof}
This follows from Proposition~\ref{prop:VecActFunc} and Proposition~\ref{prop:LeibnizRule}.
\end{proof}

\begin{Theorem}
\label{thm:CartanOnX}
Let $X$ be an object in a cartesian tangent category with scalar $R$-multiplication. Let $\calO \coloneqq \Hom(X,R)$. Then the Lie algebra of vector fields $\calX \coloneqq \Gamma(X,TX)$ and the graded commutative $\calO$-algebra of forms,
\begin{equation*}
  \Omega^\bullet(X) 
  \coloneqq 
  \Hom_{\calO}\bigl(\wedge^\bullet_{\calO} \calX, 
    \calO \bigr)
  \,,
\end{equation*}
together with the Chevalley-Eilenberg differential~\eqref{eq:differential} and the inner derivative~\eqref{eq:InnDeriv} constitute a Cartan calculus in the sense of Definition~\ref{def:CartanCalc}.
\end{Theorem}
\begin{proof}
The statement follows from applying Proposition~\ref{prop:CartanOnA} to the tangent Lie algebroid of $X$.
\end{proof}

We will call the Cartan calculus of Theorem~\ref{thm:CartanOnX} the Cartan calculus of \textbf{Lie-Rinehart forms}. In a closely related forthcoming paper, the second author will show that there is another Cartan calculus of cubical forms, defined as maps $T^n X \to R$ with additional properties. This second Cartan calculus is, in some sense, more closely related to the properties of the object $X$.

\section{Examples}
\label{sec:Examples}

We refer the reader to the companion paper \cite{AintablianBlohmann:2025} for some motivating examples (and non-examples) of cartesian tangent categories with scalar multiplication. Here, we will extend the list of examples and spell out their Cartan calculi.

\subsection{Manifolds}
\label{sec:Manifolds}

The most basic example is the tangent category of smooth manifolds. On an open subset $U \subset \bbR^n$, the functors that appear in Definition~\ref{def:TangentStructure} are given by
\begin{equation*}
\begin{aligned}
  T U &= U \times \bbR^n
  \\
  T^2 U &= U \times \bbR^n \times \bbR^n \times \bbR^n
  \\
  T^k U &= U \times (\bbR^n)^{2^k-1}
  \\
  T_2 U &= U \times \bbR^n \times \bbR^n
  \\
  T_k U &= U \times (\bbR^n)^k
  \,.
\end{aligned}
\end{equation*}
On a smooth map $f: U \to V \subset \bbR^m$ the functors are given by
\begin{align}
  T\! f: (x^i, x_1^i) &\longmapsto
  \Bigl( f^a(x), \frac{\partial f^a}{\partial x^i} x_1^i \Bigr)
  \label{eq:T1f}\\
  T^2\! f: (x^i, x_1^i, x_2^i, x_{12}^i) &\longmapsto
  \Bigl( f^a(x), 
  \frac{\partial f^a}{\partial x^i} x_1^i , 
  \frac{\partial f^a}{\partial x^i} x_2^i ,
  \frac{\partial f^a}{\partial x^i} x_{12}^i +
  \frac{\partial^2 f^a}{\partial x^i \partial x^j} x_1^i x_2^j 
  \Bigr)
  \label{eq:T2f}\\
  T_2 f: (x^i, x_1^i, x_2^i) &\longmapsto
  \Bigl( 
  f^a(x), 
  \frac{\partial f^a}{\partial x^i} x_1^i ,
  \frac{\partial f^a}{\partial x^i} x_2^i 
  \Bigr)
  \,. \notag
\end{align}
The formulas for $T^k$ and $T_k$ are analogous. The natural transformations of the tangent category structure are given by
\begin{align*}
  \pi_U : (x,x_1) &\longmapsto x
  \\
  0_U : x &\longmapsto (x, 0)
  \\
  +_U : (x, x_1, y_1) &\longmapsto (x, x_1 + y_1)
  \\
  \lambda_U : (x,x_1) &\longmapsto (x,0,0,x_1)
  \\
  \tau_U : (x, x_1, x_2, x_{12} ) &\longmapsto (x, x_2, x_1, x_{12} )
  \,.
\end{align*}
The commutativity of $T_2$ and $T$ is given by the isomorphism
\begin{equation*}
\begin{aligned}
  T(T_2 U) &\longrightarrow T_2(T U)
  \\
  \bigl( (x, x_1, y_1), (x_2, x_{12}, y_{12}) \bigr)
  &\longmapsto
  \bigl( ( x, x_2), (x_1, x_{12}), (y_1, y_{12}) \bigr) \,.
\end{aligned}
\end{equation*}
The bundle projection extends to $T^2$ as
\begin{align*}
  (\pi T)_U = \pi_{TU}: (x, x_1, x_2, x_{12}) &\longmapsto (x, x_1)
  \\
  (T\pi)_U = T\pi_U: (x, x_1, x_2, x_{12}) &\longmapsto (x, x_2)
  \,.
\end{align*}
The other natural transformations that appear in the definition, $+T$, $T+$, $\lambda T$, $T\lambda$, $\tau T$, and $T\tau$, are obtained in a similar way. The extension~\eqref{eq:VertLiftExt} of the vertical lift is given by
\begin{equation*}
  (\lambda_2)_U: (x, x_1, y_1) \longmapsto (x, x_1, 0,  y_1)
  \,.
\end{equation*}
In addition, we have the natural fiberwise multiplication by real numbers
\begin{equation*}
\begin{aligned}
  \kappa_U: \bbR \times T U 
  &\longrightarrow T U
  \\
  \bigl(r, (x,x_1) \bigr)
  &\longmapsto (x, r x_1)
  \,.
\end{aligned}
\end{equation*}
It can be checked by explicit elementary calculation that this constitutes a cartesian tangent structure with scalar $\bbR$-multiplication on the category of smooth finite-dimensional manifolds $\Mfld$.

Since the sheaf of vector fields on a manifold $M$ is a locally free and finitely generated $C^\infty(M)$-module, the Lie-Rinehart forms on $M$ can be identified with the de Rham forms. In this way, we recover the usual Cartan calculus on smooth manifolds.

\subsection{\texorpdfstring{$G$-manifolds}{G-manifolds}}
\label{sec:GManifolds}

Let $G$ be a fixed Lie group. A $G$-manifold is a smooth manifold $M$ together with a smooth left $G$-action $\rho: G \times M \to M$ \cite[Ch.~9]{Lee:2013}. A morphism is a smooth $G$-equivariant map. The tangent functor is given by $TM$ with the induced $G$-action $T_{(2)}\rho: G \times TM \to TM$. Fiber products of the tangent bundle are equipped with the diagonal $G$-action. Since all structure transformations $\pi$, $0$, $+$, $\tau$, $\lambda$, and $\kappa$ of the tangent structure are natural, they are $G$-equivariant. We conclude that $G$-manifolds form a cartesian tangent category with scalar $\bbR$-multiplication, in agreement with the general construction of MacAdam~\cite[Prop.~4.1]{MacAdam:2021} for categories of bundles over a fixed base.

The vector fields of this tangent category are the $G$-equivariant vector fields; the ring of scalar functions is $\Hom(M,\bbR) \cong C^\infty(M)$, and the Lie-Rinehart forms are the $G$-invariant de Rham forms $\Omega(M)^G$. In geometry, $G$-invariant vector fields and differential forms play an important role in symplectic reduction~\cite{MarsdenWeinstein:1974} and geometric invariant theory~\cite{Kirwan:1984}. If $G$ is compact and connected, a standard averaging argument shows that the inclusion $\Omega(M)^G \hookrightarrow \Omega(M)$ induces an isomorphism in de Rham cohomology (see, e.g., \cite[Ch.~9]{Lee:2013}). This is one reason why equivariant cohomology is defined differently in topology (for instance via the Borel construction).

\subsection{Lie groupoids}

Let us consider the category of Lie groupoids and smooth homomorphisms (see \cite[Ch.~5]{Mackenzie:2005} and \cite[Ch.~I]{MoerdijkMrcun:2003}). The tangent morphism of a submersion of smooth manifolds is a submersion. Moreover, the tangent functor commutes with pullbacks over submersions, so $T(G_1 \times_{G_0} G_1) \cong T G_1 \times_{T G_0} T G_1$. From these two facts it follows that the tangent structure of smooth manifolds induces a tangent structure on the category of Lie groupoids and smooth homomorphisms. The ring object of scalar multiplication is the manifold $\bbR$ viewed as a Lie groupoid with only identity arrows.

The vector fields of this tangent structure on a Lie groupoid $G$ are the \emph{multiplicative} vector fields~\cite[Def.~3.3]{Mackenzie:1998}, that is, the vector fields $v: G \to TG$ that are homomorphisms of Lie groupoids. The ring of scalar functions is the ring of smooth functions on the base $G_0$ that are constant on the orbits. The Lie-Rinehart forms are the multiplicative differential forms on $G$ in the sense of \cite[\S3]{Mackenzie:1998}. In this way we obtain the Cartan calculus of multiplicative forms and multiplicative vector fields.

The full subcategory generated by action groupoids for a fixed Lie group $G$ is equivalent to the category of $G$-manifolds with its tangent structure (Section~\ref{sec:GManifolds}).

\subsection{Log manifolds}

A log manifold (also called a $b$-manifold) is a pair $(M,Z)$ consisting of a smooth manifold $M$ and an embedded smooth codimension-$1$ submanifold $Z$ (see \cite[Ch.~I, \S10]{Melrose:1993} and \cite{GuilleminMirandaPires:2014}). More generally, one may take $Z$ to be a normal crossing divisor, that is, an immersed submanifold that is locally a union of embedded codimension-$1$ submanifolds intersecting transversely. The sheaf of vector fields tangent to $Z$, which we denote by $\bcalX(M,Z)$, is locally free and finitely generated as a $C^\infty(M)$-module, so it is the sheaf of sections of a vector bundle $\bT M \to M$, called the \textbf{log tangent bundle}. In local coordinates $(x,z_1,\dots,z_{n-1})$ adapted to $Z = \{x=0\}$, a frame is given by $x\partial_x, \partial_{z_1}, \dots, \partial_{z_{n-1}}$. For $x \neq 0$, we have
\begin{equation*}
  \frac{\partial f}{\partial(-\log x)}
  = \frac{\partial x}{\partial(-\log x)} \frac{\partial f}{\partial x}
  = x \frac{\partial f}{\partial x}
  \,.
\end{equation*}
This shows that $\bcalX(M,Z)$ is the Lie algebra of derivations of the ring of functions generated by the coordinates of $Z$ together with the logarithm of the normal coordinate.

The log tangent bundle $\bT M$ with the Lie bracket of vector fields is a Lie algebroid. A morphism $f: (M,Z)\to (M',Z')$ of log manifolds is a morphism of Lie algebroids $\bT M \to \bT M'$. Explicitly, a morphism is given by a smooth map $f: M\to M'$ such that $f^{-1}(Z') = Z$ and $f$ is transverse to $Z'$. We will denote the category of log manifolds by $\bMfld$. 

The restriction $\bT M |_Z \subset \bT M$ of the log tangent bundle to the submanifold $Z \subset M$ is an embedded submanifold of codimension $1$. Therefore, we obtain the endofunctor
\begin{align*}
  \bT: \bMfld &\longrightarrow \bMfld
  \\
  (M,Z) &\longmapsto (\bT M , \bT M|_Z)
  \,,
\end{align*}
which is the natural tangent functor on log manifolds. Since all axioms of a tangent category are local and, in adapted coordinates, $\bT M$ is obtained from $TM$ by replacing the vector field $\partial_x$ with $x\partial_x$, the structure maps satisfy the relations of a tangent category by the same coordinate computations as for ordinary manifolds in Section~\ref{sec:Manifolds}. In this way, we obtain a cartesian tangent category with scalar $\bbR$-multiplication. The vector fields of this tangent category are given by $\bcalX(M,Z)$. The Lie-Rinehart forms are the log forms in the usual sense \cite{GuilleminMirandaPires:2014}.

\subsection{Pro-manifolds}

An ind-object in a category $\calC$ is the colimit presheaf of a diagram $X: \calI \to \calC$ on a small filtered index category $\calI$. A morphism between ind-objects is a morphism of presheaves. The functor that maps a diagram to the colimit presheaf is a functor
\begin{equation*}
  \calC^\calI \longrightarrow \Ind(\calC)
\end{equation*}
which preserves finite limits and finite colimits. Moreover, the map that sends a category $\calC$ to its ind-category $\Ind(\calC)$ is a functor of 2-categories
\begin{equation*}
  \Ind: \Cat \longrightarrow \Cat
  \,.
\end{equation*}
It follows from the functoriality of $\Ind$ and from the fact that $\Ind$ preserves finite pointwise limits that if $\calC$ is a cartesian tangent category with scalar multiplication, then so is its ind-category $\Ind(\calC)$ (cf.~\cite{Vooys:2023}).

Dually, the pro-category of $\calC$ is given by $\Pro(\calC) = \Ind(\calC^\op)^\op$. Since $\Ind$ preserves finite pointwise colimits, $\Pro$ preserves finite pointwise limits. We therefore arrive at the following result.

\begin{Proposition}
\label{prop:IndProTan}
Let $\calC$ be a cartesian tangent category with scalar $R$-multiplication. The application of the 2-functor $\Ind$ to the tangent structure equips the ind-category of $\calC$ with the structure of a cartesian tangent category with scalar $R$-multiplication. The same statement holds for the pro-category of $\calC$.
\end{Proposition}

Concretely, the tangent functor on a pro-manifold represented by the diagram $X: \calI \to \Mfld$ is represented by the diagram $TX: \calI \to \Mfld$, $i \mapsto TX_i$. The bundle projection is given by $(\pi_X)_i \coloneqq \pi_{X_i}: TX_i \to X_i$, the addition by $(+_X)_i: TX_i \times_{X_i} TX_i \to TX_i$, etc.

\begin{Example}
The Cartan calculus of the pro-manifold of infinite jets
\begin{equation*}
  J^0 E \leftarrow J^1 E \leftarrow J^2 E
  \leftarrow \ldots
\end{equation*}
is the usual Cartan calculus of the variational bicomplex \cite{Blohmann:2024b}.
\end{Example}

\subsection{Elastic diffeological spaces}

Diffeological spaces are concrete sheaves on the site of Euclidean spaces, consisting of open subsets of $\bbR^n$, $n \geq 0$, and smooth maps, with the usual notion of open covers \cite[Def.~3.4]{Blohmann:2024a}; see also \cite[Ch.~2]{IglesiasZemmour:2013}. A diffeological space is called \emph{elastic} if it satisfies the axioms of \cite[Def.~4.1]{Blohmann:2024a} that ensure that the left Kan extension of the tangent structure on Euclidean spaces defines a tangent structure on elastic diffeological spaces with scalar $\bbR$-multiplication. The full subcategory of elastic diffeological spaces is a cartesian tangent category with scalar $\bbR$-multiplication \cite[Thm.~4.2]{Blohmann:2024a}. There are many applications of diffeological spaces, for example to Lie theory \cite{AintablianBlohmann:2025}, classical field theory \cite{Blohmann:2024b}, and general relativity \cite{BlohmannFernandesWeinstein:2013}.

\subsection{Affine schemes}
\label{sec:AffSchemes}

In \cite[\S4.1]{CruttwellLemay:2023} it was shown that the category $\Aff_k \cong \CAlg_k^\op$ of affine schemes over a field $k$ is equipped with a cartesian tangent structure. The tangent functor is defined by
\begin{equation*}
  T\Spec(A)
  \coloneqq \Spec\bigl( \Sym_A(\Omega^1_A) \bigr)
  \,,
\end{equation*}
where $\Omega^1_A$ is the $A$-module of K\"ahler differentials of the commutative $k$-algebra $A$. We recall that $\Omega^1_A$ is the $A$-module generated by the symbols $\{da ~|~ a \in A\}$ subject to the relations $d(a + b) = da + db$ and $d(ab) = (da)b + a(db)$. The bundle projection $\pi_A: T\Spec A \to \Spec A$ is given by the unit $A \hookrightarrow \Sym_A(\Omega^1_A)$. The other structure transformations of a tangent category are constructed in \cite[Lem.~4.3]{CruttwellLemay:2023}. 

A vector field is a map of algebras $v: \Sym_A(\Omega^1_A) \to A$ that is the identity on $A \subset \Sym_A(\Omega^1_A)$, that is, $A$-linear. Using that the K\"ahler differentials represent the derivations of $A$,
\begin{equation*}
  \Der_k(A, M) \cong \Hom_A(\Omega^1_A, M)
\end{equation*}
for all $A$-modules $M$, and that $\Sym_A: \Mod_A \to \CAlg_A$ is the left adjoint of the forgetful functor $\CAlg_A \to \Mod_A$, we see that the vector fields are given by
\begin{equation*}
\begin{split}
  \calX\bigl( \Spec(A) \bigr)
  &= 
  \CAlg_A\bigl( \Sym_A(\Omega^1_A), A \bigr)
  \cong
  \Hom_A(\Omega^1_A, A)
  \cong \Der_k(A,A)
  \\
  &= \Der_k(A)
  \,,
\end{split}
\end{equation*}
the Lie algebra of derivations of $A$.

For the Lie algebroid structure of the tangent bundle and the Cartan calculus, we need a scalar multiplication. The ring of scalars is the affine line $\bbA_k = \Spec(k[x])$. The scalar multiplication $\bbA_k \times T\Spec(A)  \to T\Spec(A)$ is given by the $A$-linear $k[x]$-coaction
\begin{align*}
  \kappa_A: \Sym_A(\Omega^1_A) 
  &\longrightarrow k[x] \otimes \Sym_A(\Omega^1_A)
  \\
  a 
  &\longmapsto 1 \otimes a 
  \\
  da 
  &\longmapsto x \otimes da
  \,,
\end{align*}
which defines a scalar $\bbA_k$-multiplication in the sense of Definition~\ref{def:RScalar}. The ring of $\bbA_k$-valued functions on $\Spec(A)$ is the ring itself,
\begin{equation*}
\begin{split}
  \Hom\bigl( \Spec(A), \bbA_k \bigr) 
  &= \CAlg_k\bigl( k[x], A \bigr)
  \\
  &\cong A
  \,.    
\end{split}
\end{equation*}
We conclude that the Lie-Rinehart algebra of Proposition~\ref{prop:LieAlgdRinehart} is $(A,\Der_k(A))$.

\begin{Proposition}
The Cartan calculus (Theorem~\ref{thm:CartanOnX}) of an affine scheme $\Spec(A)$ with $A$ smooth over $k$ is the usual Cartan calculus, given in terms of the Lie algebra of derivations and the $k$-algebra of K\"ahler forms on $A$.
\end{Proposition}
\begin{proof}[Sketch of proof]
When $A$ is smooth, the sheaf $\Omega_A^1$ of K\"ahler differentials is locally a finitely generated projective $A$-module. In particular, it is reflexive, which implies that
\begin{equation*}
\begin{split}
  \Hom_A\bigl( \Der_k(A), A \bigr) 
  \cong \Hom_A\bigl( \Hom_A(\Omega^1_A, A), A)
  \cong \Omega^1_A
  \,.    
\end{split}
\end{equation*}
This generalizes to Lie-Rinehart forms of degree $n > 1$, 
\begin{equation*}
\begin{split}
  \Hom_A\bigl( \wedge^n_A \Der_k(A), A \bigr) 
  \cong \Omega^n_A
  \,,
\end{split}
\end{equation*}
which shows that they are the usual de Rham forms on $\Spec(A)$.
\end{proof}

\subsection{Schemes}
\label{sec:Schemes}

Let $X$ be a $k$-scheme. Choosing a Zariski open cover $\{U_i\}$ of $X$ by smooth affines $U_i = \Spec(A_i)$, Section~\ref{sec:AffSchemes} provides a Cartan calculus on every $U_i$. On overlaps $U_i \cap U_j$ these Cartan calculi are compatible with restriction, hence they glue to a Cartan calculus on $X$ at the level of sheaves. The tangent sheaf $\calT_{X/k} \coloneqq \Der_k(\calO_X)$ is the sheaf of $k$-derivations of the structure sheaf $\calO_X$ \cite[Ch.~II, \S8]{Hartshorne:1977}. 

When $X$ is smooth, the sheaf of Lie-Rinehart forms can be identified with the sheaf of de Rham forms
\begin{equation*}
\begin{split}
  \Hom_{\calO_X}\bigl( \wedge^n_{\calO_X} \Der_k(\calO_X), \calO_X \bigr) 
  \cong \Omega^n_{X/k}
  \,.    
\end{split}
\end{equation*}
For smooth schemes, the induced differential, contractions, and Lie derivatives on $\Omega^\bullet_{X/k}$ recover the standard Cartan calculus on de Rham forms. The algebraic de Rham cohomology is the hypercohomology of the sheaf,
\begin{equation*}
  H_{\mathrm{dR}}(X/k)
  \cong
  \mathbb{H}\bigl(X, \Omega_{X/k} \bigr)
  \,,
\end{equation*}
which is the hypercohomology of the de Rham complex of sheaves \cite[Ch.~III, Ex.~7.7]{Hartshorne:1977}. Cruttwell and Lemay~\cite[Thm.~4.28]{CruttwellLemay:2023} show that differential bundles over a scheme correspond to quasi-coherent sheaves; our Lie-Rinehart forms provide the Cartan calculus built from the tangent sheaf $\calT_{X/k}$ rather than from differential bundles.

\subsection{Graded manifolds}
\label{sec:GradedManifolds}

Let $G$ be an abelian group, typically $\bbZ_2$, $\bbZ$ or products thereof. Let $V = \{V_g ~|~ g \in G\}$ be a $G$-graded $\bbR$-vector space of total dimension $\sum_{g \in G} \dim V_g < \infty$. A $G$-graded manifold modeled on $V$ is a pair $\calM = (M, \calO_\calM)$ consisting of a manifold $M$ and a sheaf $\calO_\calM$ of graded commutative $\bbR$-algebras that is locally of the form
\begin{equation*}
  \calO_\calM(U)
  \cong C^\infty(U) \otimes \SymGr(V^*)
\end{equation*}
on the open subsets $U \subset M$ of a cover, where $\SymGr(V^*)$ is the free graded commutative $\bbR$-algebra generated by $V^*$ \cite{Roytenberg:2002,Fairon:2017}. A morphism $\calM = (M, \calO_\calM) \to (N, \calO_\calN) = \calN$ of graded manifolds is a morphism of ringed spaces, that is, a smooth map of manifolds $\phi: M \to N$ together with a morphism of sheaves of graded algebras $\phi^\sharp: \calO_\calN \to \phi_* \calO_\calM$. It is implied that the zero component of $\phi^\sharp$ is the pullback of smooth functions on $N$ by $\phi$.

The tangent functor on graded manifolds is defined by
\begin{equation*}
  T(M, \calO_\calM) = (TM, \calO_{T\calM})
  \,,
\end{equation*}
where $TM$ is the tangent manifold and the structure sheaf is given by
\begin{equation*}
  \calO_{T\calM}(TU)
  = 
  C^\infty(TU) \otimes \SymGr_{\SymGr(V^*)}\bigl(\Omega^1_{\SymGr(V^*)} \bigr)
  \,.
\end{equation*}
The difference to affine schemes is that the degree zero subalgebra is given by the algebra of smooth functions rather than the free algebra generated by the K\"ahler differentials,
\begin{equation*}
  (\calO_{T\calM})_0 
  = C^\infty(TU) 
  \supsetneq \SymGr_{C^\infty(U)}\bigl(\Omega^1_{C^\infty(U)} \bigr)
  \,.
\end{equation*}

Since $\SymGr(V^*)$ is freely generated by $V^*$, the $\SymGr(V^*)$-module of K\"ahler differentials is freely generated by $dV^* \coloneqq \Span\{d\theta~|~ \theta \in V^*\}$, so we obtain
\begin{equation*}
\begin{split}
  \SymGr_{\SymGr(V^*)}\bigl(\Omega^1_{\SymGr(V^*)} \bigr)
  &\cong
  \SymGr(V^*) \otimes \SymGr(dV^*)
  \cong \SymGr\bigl( (V\oplus V)^* \bigr)
  \\
  &\cong \SymGr\bigl( (TV)^* \bigr)
  \,,
\end{split}
\end{equation*}
where $TV = V \oplus V$ is the tangent functor of graded vector spaces. A tangent morphism is given by the pair $T(\phi,\phi^\sharp) = (T\phi, T\phi^\sharp)$ where $T\phi: TM \to TN$ is the tangent map of manifolds and
\begin{align*}
  T\phi^\sharp: \calO_{T\calM}(TU)
  &\longrightarrow \calO_{T\calN}\bigl((T\phi)^{-1}(TU)\bigr)
  \\
  f \otimes da
  &\longmapsto 
  (f\circ T\phi) \otimes d \bigl( \phi^\sharp a \bigr)
\end{align*}
for all $f \in C^\infty(TU)$ and all generators $da \in \Omega^1_{\SymGr(V^*)}$ of $\SymGr_{\SymGr(V^*)}\bigl(\Omega^1_{\SymGr(V^*)} \bigr)$.

The natural transformations of the tangent structure are obtained as pairs of the natural transformations of manifolds and the graded version of those for affine schemes. For example, the bundle projection is given on the structure algebras by
\begin{align*}
  \pi^\sharp_\calM:
  \calO_{\calM}(U)
  &\longrightarrow
  \calO_{T\calM}\bigl( \pi_M^{-1}(U) \bigr)
  \\
  f \otimes a
  &\longmapsto
  \pi_M^* f \otimes a
  \,,
\end{align*}
where $a \in \SymGr(V^*)$. The other natural transformations are similarly straightforward graded generalizations of those for affine schemes.

The ring object of the scalar multiplication is $\bbR$ viewed as graded manifold with structure sheaf $\calO_\bbR(I) = C^\infty(I)$ for any open interval $I \subset \bbR$. The product of graded manifolds is defined by $\calM \times \calN = (M \times N, \calO_\calM \boxtimes \calO_\calN)$, so
\begin{equation*}
  \calO_{\bbR \times T\calM}(I \times TU)
  = C^\infty(I \times TU) \otimes 
  \SymGr_{\SymGr(V^*)}\bigl(\Omega^1_{\SymGr(V^*)} \bigr)
  \,.
\end{equation*}
Let $x \coloneqq \id_\bbR : \bbR \to \bbR$ denote the tautological chart of $\bbR$ and $\pr_1^*$ its pullback to $\bbR \times TU$ by the projection to $\bbR$. The scalar multiplication is given on the structure sheaf by
\begin{align*}
  \kappa^\sharp_\calM: 
  \calO_{T\calM}(TU) 
  &\longrightarrow
  \calO_{\bbR \times T\calM}
  \bigl(\kappa_M^{-1}(I \times TU) \bigr)
  \\
  f \otimes 1 
  &\longmapsto (\kappa_M^* f) \otimes 1 
  \\
  1 \otimes a 
  &\longmapsto 1 \otimes a 
  \\
  1 \otimes da 
  &\longmapsto \pr_1^* x \otimes da
  \,,
\end{align*}
for all $f \in C^\infty(TU)$ and all generators $da \in \Omega^1_{\SymGr(V^*)}$ of $\SymGr_{\SymGr(V^*)}\bigl(\Omega^1_{\SymGr(V^*)} \bigr)$.

The sheaf of vector fields on $\calM$ is the sheaf of degree 0 derivations of the structure algebra. On a chart $U \subset M$, where we have $TU \cong U \times \bbR^n$, it is given by
\begin{equation*}
\begin{split}
  \calX_\calM(U)
  &\cong \Der\bigl( \calO_\calM(U) \bigr)
  \\
  &\cong \Der\bigl( C^\infty(U) \otimes \SymGr(V^*) \bigr)
  \\
  &\cong 
    \bigl\{ \Der\bigl( C^\infty(U) \bigr) 
      \otimes \SymGr(V^*) \bigr\}
    \oplus \bigl\{ C^\infty(U) 
    \otimes \Der\bigl(\SymGr(V^*) \bigr) \bigr\}
  \\
  &\cong 
    \bigl\{ \Gamma(U,U \times \bbR^n) \otimes \SymGr(V^*) \bigr\}
    \oplus \bigl\{ C^\infty(U) 
    \otimes \Hom_\bbR\bigl(V^*, \SymGr(V^*) \bigr) \bigr\}
  \\
  &\cong 
    \bigl\{C^\infty(U) \otimes \bbR^n \otimes \SymGr(V^*) \bigr\}
    \oplus \bigl\{ C^\infty(U) 
    \otimes V \otimes \SymGr(V^*) \bigr\}
  \\
  &\cong \calO_\calM(U) \otimes (\bbR^n \oplus V)
  \,,
\end{split}
\end{equation*}
where we have used that the total dimension of $V$ is finite, so that $\Hom_\bbR(V^*,W) \cong V \otimes W$. We see that the sheaf of vector fields is locally free over the structure sheaf.

The global vector fields are the derivations of the algebra $\calO_\calM(M)$. This recovers the usual concept of degree 0 vector fields on a graded manifold. The Lie-Rinehart forms are locally given by the sheaf
\begin{equation}
\label{eq:OmegaGradMan}
  \Omega^k_\calM(U)
  = \bigoplus_{p+q=k}\Omega^p(U) \otimes \wedge^q V
  \,,
\end{equation}
where $\wedge^q V \cong \SymGr(V[1])[-q]$ is the graded antisymmetric algebra generated by $V$. The global forms are $\Omega_\calM(M)$, where $k$ in \eqref{eq:OmegaGradMan} is the \emph{form degree}. The sheaf $\Omega_\calM$ of differential forms on $\calM$ is isomorphic to the structure ring of the $1$-shifted tangent bundle
\begin{equation*}
  \Omega_\calM \cong \calO_{T[1]\calM}
  \,.
\end{equation*}
This is the standard identification between differential forms and functions on the shifted tangent bundle.

\begin{Remark}
The vector fields of the tangent structure on graded manifolds are all of degree $0$. To obtain graded vector fields, one has to enrich the category over graded vector spaces; see \cite{Roytenberg:2002} for the $\bbZ_2$-graded case. The notion of a tangent structure on an enriched category seems rather straightforward to define, but to the best of our knowledge this has not been studied in the literature, so it would be a natural direction for future research.
\end{Remark}

\subsection{\texorpdfstring{Affine $C^\infty$-schemes}{Affine C-infinity-schemes}}

Let $\Eucl$ denote the category with open subsets of all $\bbR^n$, $n \geq 0$ as objects and smooth maps as morphisms. By definition, a $C^\infty$-ring is a product-preserving functor $\Eucl \to \Set$ \cite[Def.~2.4]{Joyce:2019}; see also \cite{Dubuc:1981,MoerdijkReyes:1991}. Explicitly, a $C^\infty$-ring $A$ is a set with an $n$-ary operation $A^n \to A$ for every smooth map $f: \bbR^n \to \bbR$ and all $n \geq 0$, satisfying operadic associativity. Since the morphisms in $\Eucl$ include the commutative ring structure of $\bbR$, every $C^\infty$-ring is a commutative ring. We denote the coproduct of $C^\infty$-rings $A$ and $B$ by $A \otimes_\infty B \equiv A \sqcup B$. 

\begin{Example}
Let $C^\infty(M)$ be the set of smooth functions on a manifold $M$. Any smooth map $f: \bbR^n \to \bbR$ maps an $n$-tuple $(a_1, \ldots, a_n) \in C^\infty(M)^n$ to the smooth function $f(a_1, \ldots, a_n) \coloneqq f \circ (a_1, \ldots, a_n)$. This equips $C^\infty(M)$ with the structure of a $C^\infty$-ring. The coproduct of two such rings is given by
\begin{equation*}
  C^\infty(M) \otimes_\infty C^\infty(N) 
  \cong C^\infty(M \times N)
  \,.
\end{equation*}
\end{Example}

Let us denote the category of $C^\infty$-rings by $\CinftyRing$. 
The opposite category $\CinftyRing^\op$ is the category of affine $C^\infty$-schemes \cite[\S4]{Joyce:2019}. The $C^\infty$-scheme given by the $C^\infty$-ring $A$ is denoted $\Spec^\infty(A)$. The tangent structure on $\CinftyRing^\op$ is the $C^\infty$-analog of that for affine schemes (Section~\ref{sec:AffSchemes}). Let $A$ be a $C^\infty$-ring. The $A$-module of \textbf{$C^\infty$-differentials} $\Omega^{\infty,1}_A$ is the cotangent module of $A$ in the sense of \cite[Def.~5.3]{Joyce:2019}: the free $A$-module generated by the symbols $da$ for all $a \in A$ modulo the relations imposed by the chain rule for $C^\infty$-derivations,
\begin{equation*}
  d\bigl(f(a_1, \ldots, a_n) \bigr)
  = \sum_{i=1}^n 
  \frac{\partial f}{\partial x^i} (a_1, \ldots, a_n) \cdot da_i
  \,,
\end{equation*}
for all smooth $f: \bbR^n \to \bbR$ and $a_i \in A$. Let $\Sym_A^\infty\bigl(\Omega^{\infty,1}_A)$ denote the free $C^\infty$-ring over $A$, generated by the $A$-module of $C^\infty$-differentials. The tangent functor is given by
\begin{equation*}
  T \Spec^\infty(A) 
  \coloneqq \Spec^\infty \bigl( 
     \Sym_A^\infty(\Omega^{\infty,1}_A) \bigr)
  \,.
\end{equation*}
The natural transformations of the tangent structure are the $C^\infty$-versions of those for ordinary affine schemes.

The ring object in $\CinftyRing^\op$ of scalar multiplication is $R \coloneqq \Spec^\infty\bigl( C^\infty(\bbR) \bigr)$. The coring structure on $C^\infty(\bbR)$ is given by the pullback of the ring structure on $\bbR$. For example, the comultiplication is given by the morphism of $C^\infty$-rings
\begin{align*}
  C^\infty(\bbR)
  &\longrightarrow
  C^\infty(\bbR) \otimes_\infty C^\infty(\bbR)
  \cong C^\infty(\bbR \times \bbR)
  \\
  f &\longmapsto 
  \bigl( (x,y) \mapsto f(xy) \bigr)
  \,,
\end{align*}
etc. The scalar multiplication is given by
\begin{align*}
  \kappa_A: \Sym^\infty_A\bigl(\Omega^{\infty,1}_A\bigr)
  &\longrightarrow
  C^\infty(\bbR) \otimes_\infty \Sym^\infty_A\bigl(\Omega^{\infty,1}_A\bigr)
  \\
  a
  &\longmapsto 1 \otimes_\infty a
  \\
  da
  &\longmapsto \id_\bbR \otimes_\infty da
  \,,
\end{align*}
in analogy to the case of ordinary affine schemes. As a $C^\infty$-ring, $C^\infty(\bbR)$ is freely generated by $\id_\bbR$, so that
\begin{equation*}
\begin{split}
  \calO(A)
  &\coloneqq
  \Hom\bigl(\Spec^\infty(A),R\bigr)
  =
  \CinftyRing\bigl( C^\infty(\bbR), A \bigr)
  \cong
  \Hom_\Set( \id_\bbR, A)
  \\
  &\cong A
  \,.    
\end{split}
\end{equation*}

A \textbf{$C^\infty$-derivation} is a map $v: A \to A$ that satisfies
\begin{equation*}
  v\bigl(f(a_1, \ldots, a_n) \bigr)
  = \sum_{i=1}^n 
  \frac{\partial f}{\partial x^i} (a_1, \ldots, a_n) \cdot v(a_i)
  \,.
\end{equation*}
We denote the $C^\infty$-derivations by $\Der^\infty(A)$. It can be shown that
\begin{equation*}
  \Der^\infty(A) \cong \Hom_A(\Omega^{\infty,1}_A, A)
  \,,
\end{equation*}
in complete analogy to ordinary schemes. It follows that the vector fields are the derivations $\calX(A) \cong \Der^\infty(A)$.

When $A = C^\infty(M)$ for a manifold $M$, the Lie-Rinehart forms are the $C^\infty$-de Rham forms on $\Spec(A)$ \cite[Ex.~5.4]{Joyce:2019}, and we recover the known Cartan calculus on smooth manifolds viewed as affine $C^\infty$-schemes. In general, however, the Lie-Rinehart forms of Theorem~\ref{thm:CartanOnX} are quite different from the $C^\infty$-cotangent complex: Joyce's cotangent module $\Omega_A$ is defined by the universal property for $C^\infty$-derivations \cite[Def.~5.3]{Joyce:2019}, whereas our Lie-Rinehart forms are Chevalley-Eilenberg forms built from the Lie algebra of $C^\infty$-derivations.

\subsection{Differentiable stacks}
\label{sec:DifferentiableStacks}

The category of differentiable stacks \cite{BehrendXu:2011} is a $2$-category, so, strictly speaking, it is outside the scope of this paper, but within the scope of the framework of higher tangent categories developed by Kristine Bauer, Matthew Burke, and Michael Ching \cite{BauerBurkeChing:2021}. We would still like to indicate how, conjecturally, differentiable stacks are equipped with the structure of a $(2,1)$-tangent category.

The category of differentiable stacks is equivalent to the $(2,1)$-category of Lie groupoids, right principal groupoid bibundles, and biequivariant diffeomorphisms \cite{Blohmann:2008}. In this category, the usual methods of differentiation on smooth manifolds are at our disposal. By applying the tangent functor to a Lie groupoid $G$, we obtain the tangent groupoid $TG$. By applying the tangent functor to a $G$-$H$ bibundle $P$, we obtain a $TG$-$TH$ bibundle $TP$. The bibundle $P$ is $H$-principal if and only if the endpoint map of the $H$-action, $P \times_{H_0} H_0 \to P \times_{G_0} P$ is an isomorphism. As is the case for any functor, $T$ preserves isomorphisms, so if $P$ is right principal, then $TP$ is right principal. Similarly, the natural transformations of the tangent structure on manifolds induce natural transformations on Lie groupoids and bibundles. We posit that this equips the category of differentiable stacks with a higher tangent structure. Verifying this rigorously is a task for future research.

\bibliographystyle{alpha}
\bibliography{Cartan}

\end{document}